\documentclass[a4paper, notitlepage, oneside, reqno]{amsart}
\usepackage{lmodern} 
\usepackage{amsmath,amsfonts,amssymb}
\usepackage{amsthm} 
\usepackage{graphicx} 
\usepackage{subfig}
\usepackage{listings}
\usepackage{geometry}
\usepackage{hyperref}
\usepackage{wrapfig}
\usepackage{url}
\usepackage{tikz}
\usetikzlibrary{shapes, positioning, arrows}
\usepackage{eso-pic}
\usepackage{mathrsfs}
\usepackage{xcolor}
\usepackage{mathtools}
\usepackage{multirow}

\newcommand{\dx}{\, \mathrm{d}x}

\newcommand{\dxi}{\, \mathrm{d}\xi}

\newcommand{\R}{\mathbb{R}}
\newcommand{\Z}{\mathbb{Z}}
\newcommand{\N}{\mathbb{N}}

\newcommand{\Q}{\mathcal{Q}}

\newcommand{\E}{\mathcal{E}}
\newcommand{\J}{\mathcal{J}}

\newcommand{\U}{\mathcal{U}}

\newtheorem{theorem}{Theorem}[section]
\newtheorem{lemma}[theorem]{Lemma}
\newtheorem{corollary}[theorem]{Corollary}

\theoremstyle{definition}

\newtheorem{remark}[theorem]{Remark}
\numberwithin{equation}{section}

\title{Existence of solitary-wave solutions to nonlocal equations}
\author{Mathias Nikolai Arnesen}
\date{}
\thanks{The author gratefully acknowledges the support of the project Nonlinear water waves (Grant No. 231668) from the Research Council of Norway.}

\begin{document}

\begin{abstract}
We prove existence and conditional energetic stability of solitary-wave solutions for the two classes of pseudodifferential equations 
\begin{equation*}
u_t+\left(f(u)\right)_x-\left(L u\right)_x=0
\end{equation*}
and
\begin{equation*}
u_t+\left(f(u)\right)_x+\left(L u\right)_t=0,
\end{equation*}
where $f$ is a nonlinear term, typically of the form $c|u|^p$ or $cu|u|^{p-1}$, and $L$ is a Fourier multiplier operator of positive order. The former class includes for instance the Whitham equation with capillary effects and the generalized Korteweg-de Vries equation, and the latter the Benjamin-Bona-Mahony equation. Existence and conditional energetic stability results have earlier been established using the method of concentration-compactness for a class of operators with symbol of order $s\geq 1$. We extend these results to symbols of order $0<s<1$, thereby improving upon the results for general operators with symbol of order $s\geq 1$ by enlarging both the class of linear operators and nonlinearities admitting existence of solitary waves. Instead of using abstract operator theory, the new results are obtained by direct calculations involving the nonlocal operator $L$, something that gives us the bounds and estimates needed for the method of concentration-compactness.
\end{abstract}

\maketitle

\section{Introduction}
\label{introduction}
In this paper we discuss solitary-wave solutions of pseudodifferential equations of the form
\begin{equation}
\label{eq: main equation}
u_t+\left(f(u)\right)_x-\left(L u\right)_x=0
\end{equation}
or
\begin{equation}
\label{eq: main equation 2}
u_t+\left(f(u)\right)_x+\left(L u\right)_t=0,
\end{equation}
where $u$ and $f$ are real-valued functions, and $L$ is a Fourier multiplier operator with symbol $m$ of order $s>0$. That is,
\begin{equation*}
\widehat{L u}(\xi)=m(\xi)\widehat{u}(\xi),
\end{equation*}
where the hat denotes the Fourier transform $\widehat{f}(\xi)=\int_{\R} \mathrm{e}^{-2\pi ix\xi} f(x) \dx$ with respect to the spatial coordinate, and $m$ is a function satisfying
\begin{align*}
& A_1 |\xi|^s\leq m(\xi)\leq A_2|\xi|^s, \quad |\xi|\geq 1, \\
& 0\leq m(\xi)\leq A_2, \quad |\xi|\leq 1,
\end{align*}
for some constants $A_1,A_2>0$. Our inspiration comes from \cite{EGW} and a series of recent papers on nonlinear dispersive equations with weak \cite{Linares2014dpo} or very weak \cite{Lannes2014rot} dispersion. This includes investigations into existence \cite{EEP}, stability \cite{Johnson2014mii,Klein2008ama} and travelling waves \cite{Ehrnstrom2013gbf} of the Whitham equation. We mention here that our results yield existence of solitary-waves to the capillary Whitham equation (see \cite{Lannes2014rot}), a case not earlier covered in the literature \cite{albert,zeng,guo}.

A solitary-wave is a travelling wave of the form $u(x,t)=u(x-ct)$, where $c>0$ is the speed of the wave moving from left to right, that vanishes as $x-ct\rightarrow \pm \infty$. Assuming that $u$ is a solitary-wave solution of \eqref{eq: main equation} or \eqref{eq: main equation 2}, we obtain the following equations by integrating \eqref{eq: main equation} or \eqref{eq: main equation 2}, respectively, with respect to the spatial variable:
\begin{equation}
\label{eq: main equation2}
L u+cu-f(u)=0
\end{equation}
and
\begin{equation}
\label{eq: main equationZ}
c(L u+u)-f(u)=0.
\end{equation}
For studying existence and stability of solutions to \eqref{eq: main equation2} and \eqref{eq: main equationZ} by variational methods one can consider constrained variational problems (see equations \eqref{eq: constrained var. prob. 1} and \eqref{eq: Gamma}). The loss of compactness that results from working in the unbounded domain $\R$ is overcome by the method of concentration-compactness as introduced in \cite{lions}. The main challenge in applying the concentration-compactness method is usually, particularly in the nonlocal case, to preclude dichotomy (cf. Lemma \ref{Lions}), for which one needs a result like Theorem \ref{thm: dichotomy convergence} to hold for the operator $L$.


Albert, Bona and Saut \cite{abs} prove existence of solitary-wave solutions to the Kubota-Ko-Dobbs equation, which belongs to the class of equations \eqref{eq: main equation} with an operator of order $s=1$, and their approach is presented in a more general form in \cite{albert}. The results are remarked to hold for any nonlinearity $f(u)=|u|^p$, $p\in (1,2s+1)$ or $f(u)=u^p$, $p\in \N\cap (1,2s+1)$ and any operator $L$ with symbol $m$ of order $s\geq 1$ satisfying 
\begin{equation}
\label{eq: commutator}
\|L (\theta f)-\theta L (f)\|_{L^2}\leq C\|\theta'\|_{L^\infty}\|f\|_{L^2},
\end{equation}
for any function $\theta$ and $f\in C_0^\infty$. Using general commutator estimates (\cite[Theorem 35]{coifman}), \eqref{eq: commutator} leads the authors of \cite{abs} and \cite{albert} to impose the condition
\begin{equation}
\label{eq: abs condition on m}
\left|\left(\frac{\mathrm{d}}{\mathrm{d}\xi}\right)^n\left( \frac{m(\xi)}{\xi}\right) \right|\leq C|\xi|^{-n} \quad \text{for all} \,\, \xi\in \R \,\, \text{and} \,\, n\in \N,
\end{equation}
for some constant $C>0$. This condition is never satisfied when $s>1$ or $0<s<1$. By a splitting argument, Zeng \cite{zeng} establishes a similar inequality to \eqref{eq: commutator} for all operators with symbol $m$ such that \eqref{eq: abs condition on m} is satisfied for $n\in \lbrace 0,1,2,3,4\rbrace$ with $m(\xi)/\xi$ replaced by $(m(\xi)-m(0))/\xi$ when $|\xi|\leq 1$ and by $\sqrt{m(\xi)}/|\xi|^{s/2}$ when $|\xi|\geq 1$. This also excludes symbols of order $0<s<1$, but allows one to consider operators of order $s>1$, for instance the fractional Laplace operator $(-\Delta)^{s/2}$ where $m(\xi)=|\xi|^s$. Zeng \cite{zeng} does this for equations of the form \eqref{eq: main equationZ} for nonlinearities satisfying Assumption (B) (see the Assumptions below), but this argument can easily be implement in the method of \cite{albert} to extend the results of that paper for \eqref{eq: main equation2} to operators satisfying the assumptions of \cite{zeng}.

For pseudodifferential operators of order $0<s<1$, however, the only known result is, to the author's knowledge, the recent publication \cite{guo}, which proves existence of solitary-wave solutions to \eqref{eq: main equation} for $L=(-\Delta)^{s/2}$ ($m(\xi)=|\xi|^s$) and $f(u)=c_p u|u|^{p-1}$, where $p\in (1,2s+1)$. That result was achieved by using a commutator estimate that has only been established for the fractional Laplace operator. The authors of \cite{lenzmann} remark that the method of Weinstein \cite{weinstein}, used to prove solitary-wave solutions of \eqref{eq: main equation} and \eqref{eq: main equation 2} when $s\geq 1$, holds equally well when $0<s<1$. While it is true that the method in \cite{weinstein} can be modified to prove the existence for $s\in (0,1)$ (which is one of the results in this paper), as noted in \cite{abs} and proved in the Appendix, further care needs to be taken to the nonlocal part of the problem than what is done in \cite{weinstein}; in particular, equation (3.20) in \cite{weinstein} does not hold in general. In this paper we will establish Theorem \ref{thm: dichotomy convergence} by direct calculation without any reference to general results on commutator estimates. This allows us to treat all operators of any order $s>0$ that satisfies natural and easy to check assumptions (see Assumption (A)). Moreover, in the Appendix we prove that our assumptions are, almost if not completely, as weak as they can be: under weaker assumptions, the method of concentration compactness cannot be applied.

The structure of the paper is as follows: In Section \ref{results} we state and describe our assumptions and results in detail. The main result on the existence of solitary-wave solutions, Theorem \ref{thm: existence}, will be proved in three parts, using the method of concentration-compactness, in Sections \ref{existence}, \ref{existence2} and \ref{different nonlinearities}. In Section \ref{stable regular}, Theorem \ref{thm: stability} concerning the stability of the sets of solutions is proved, as well as a result on the regularity of solutions. Lastly, in the Appendix, we prove by counter-example the necessity of a continuity assumption on the symbol $m$ in order to obtain compactness from the concentration-compactness method. The general outline of the procedure in Sections \ref{existence}, \ref{existence2} and \ref{different nonlinearities} is inspired primarily by \cite{zeng}, and also by \cite{albert}. While \cite{zeng} works only with \eqref{eq: main equation 2} and \cite{albert} with \eqref{eq: main equation}, we will relate the variational formulation \eqref{eq: Gamma} of \eqref{eq: main equationZ} to solutions of \eqref{eq: main equation2} using a scaling argument from \cite{weinstein}. This allows us to extend the range of nonlinearities for which we have existence of solutions to \eqref{eq: main equation2}.

For $1\leq p\leq \infty$ and measurable sets $\Omega\subseteq \R$ we will write $L^p(\Omega)$ for the usual Banach spaces with norm $\|f\|_{L^p(\Omega)}=\left( \int_{\Omega}|f|^p\dx\right)^{1/p}$ if $1\leq p<\infty$, and $\|f\|_{L^\infty (\Omega)}=\mathrm{ess} \,\mathrm{supp}_{x\in \Omega}|f(x)|$. The ambient space is always $\R$ and we will for convenience write $L^p$ for $L^p(\R)$. Similarly, we denote by $H^s$ the Hilbert space $H^s(\R)$ with norm $\|f\|_{H^s}=\left( \int_{\R}(1+|\xi|^2)^s|\widehat{f}(\xi)|^2\dxi\right)^{1/2}$.

\section{Assumptions and main results}
\label{results}
In this section we will state our fundamental assumptions and describe our results. The precise, technical details of our results are contained in Theorems \ref{thm: existence} and \ref{thm: stability}, while a simpler summary of which nonlinearities we have existence and stability of solitary-wave solutions for is given in Table \ref{table 1}.
\begin{itemize}
\item[(A)] The operator $L$ is a Fourier multiplier with symbol $m$ of order $s>0$. That is,
\begin{equation*}
\widehat{L u}(\xi)=m(\xi)\widehat{u}(\xi),
\end{equation*}
where $m$ satisfies
\begin{align}
\label{assumption on m}
& A_1|\xi|^s\leq m(\xi)\leq A_2|\xi|^s \,\, \text{for} \,\, |\xi|\geq 1, \nonumber \\
& 0\leq m(\xi)\leq A_2 \,\, \text{for} \,\, |\xi|\leq 1,
\end{align}
for some constants $A_1,A_2>0$. Furthermore, we assume that $m$ is piecewise continuous with finitely many discontinuities and that there exists a $K>0$ such that for all $|\xi|>K$ and $|t|\ll 1$ such that $m$ is continuous on $(\xi-t,\xi)$,
\begin{equation}
\label{assumption on m 2}
|m(\xi)-m(\xi-t)|\leq |k(t)||\xi|^s,
\end{equation}
where $\lim_{t\rightarrow 0} k(t)=0$.
\item[(B)] The nonlinearity $f$ is of one of the forms:
\begin{itemize}
\item[(B1)] $f(u)=c_p u|u|^{p-1}$ where $c_p> 0$,
\item[(B2)] $f(u)=c_p |u|^p$ where $c_p\neq 0$,
\end{itemize}
where either $p\in (1,2s+1)$ or $p\in (1,\frac{1+s}{1-s})$. When $s\geq 1$, $p\in (1,\frac{1+s}{1-s})$ should be interpreted as $p\in (1,\infty)$.
\end{itemize}
The assumption \eqref{assumption on m} in (A) is to ensure that $\left(\int_{\R} u L u +u^2\dx\right)^{1/2}$ is an equivalent norm to the standard norm on $H^{s/2}$. The continuity assumption is essential for proving Theorem \ref{thm: dichotomy convergence} which is necessary in order to exclude dichotomy, and in the Appendix we will show that a continuity assumption is necessary. The two different forms (B1) and (B2) of the nonlinearity are considered to cover both the case when the sign of $u$ does affect the sign of $f$ and when it does not, generalizing the cases when $p$ is, respectively, an odd or an even integer. The two ranges of $p$ are related to stability and existence. For e.g. the generalized Korteweg-de Vries equation (where $s=1$), it is known that $p=2s+1$ is the critical exponent beyond which one loses stability, while one has existence for all $p\in (1,\infty)$ (see, for instance, \cite{Angulo2009nde}).

To state our results, let $F$ be the primitive of $f$. That is,
\begin{equation}
\label{def: F}
F(x):=\left\{
	\begin{array}{l l}
		c_p\frac{|x|^{p+1}}{p+1}, & \quad \text{if}\,\, f(x)=c_px|x|^{p-1},\\
		c_p\frac{x|x|^p}{p+1}, & \quad \text{if}\,\, f(x)=c_p|x|^p.
	\end{array} \right.
\end{equation}
As one can check (or see e.g. Lemma 1 in \cite{abh}), if $u$ solves equation \eqref{eq: main equation} with initial condition $u(x,0)=\psi(x)$ for all $x\in \R$ where $\psi\in H^r$, $r\geq s/2$, the functionals
\begin{equation}
\label{eq: E}
\E(u)=\frac{1}{2}\int_{\R}uL u\dx-\int_{\R}F(u)\dx
\end{equation}
and
\begin{equation*}
\Q(u)=\frac{1}{2}\int_{\R}u^2\dx
\end{equation*}
are independent of $t$. Likewise, the functionals
\begin{equation*}
\J(u)=\frac{1}{2}\int_{\R} \left(uL u+u^2\right)\dx
\end{equation*}
and
\begin{equation*}
\U(u)=\int_{\R}F(u)\dx
\end{equation*}
are invariant in time for solutions of equation \eqref{eq: main equation 2}. Furthermore, the Lagrange multiplier principle (cf. \cite{zeidler}) implies that, for every $q>0$, minimizers of the constrained variational problem
\begin{equation}
\label{eq: constrained var. prob. 1}
I_q:=\inf\lbrace \E(w) : w\in H^{s/2} \,\, \text{and}\,\, \mathcal{Q}(w)=q\rbrace
\end{equation}
solve equation \eqref{eq: main equation2} with $c$ being the Lagrange multiplier. We denote by $D_q$ the set of minimizers of $I_q$. Equation \eqref{eq: constrained var. prob. 1} is the variational problem studied in \cite{albert} for a class of symbols with $s=1$, and as we shall show the results of \cite{albert} can be extended to hold for all operators satisfying assumption (A). This formulation, however, has the disadvantage that $I_q$ is unbounded below when $p\geq 2s+1$, and so the range of $p$ for which one can find minimizers is restricted to $(1,2s+1)$. One would expect a change in behaviour at the critical exponent $p=2s+1$, as with the GKdV equation as mentioned above, but one would also expect existence, if not stability, beyond the critical exponent. This is indeed the case, as we will show. For any $\lambda>0$, equation \eqref{eq: main equationZ} is the Euler-Lagrange equation of the constrained variational problem
\begin{equation}
\label{eq: Gamma}
\Gamma_\lambda= \inf \lbrace \J(w) : w\in H^{s/2} \,\, \text{and} \,\, \U(w)=\lambda\rbrace.
\end{equation}
For all $p\in (1, \frac{1+s}{1-s})$, one can show that $\Gamma_{\lambda}$ is well defined. The Lagrange multiplier principle implies that if $u$ is a minimizer of $\Gamma_\lambda$, then there exists a $\gamma$ such that
\begin{equation*}
L u+u -\gamma f(u)=0
\end{equation*}
in a weak sense, here meaning that
\begin{equation*}
\int_{\R} \left(L u+u+\gamma f(u)\right)\varphi\dx=0
\end{equation*}
for all $\varphi\in H^{s/2}$. Hence $u$ solves equation \eqref{eq: main equationZ} with $c=1/\gamma$. If we define
\begin{equation}
\label{eq: J_kappa}
\J_{\kappa}=\frac{1}{2}\int_{\R} uL u + \kappa u^2\dx=\frac{1}{2}\int_\R (m(\xi)+\kappa)|\widehat{u}(\xi)|^2\dxi,
\end{equation}
where $\kappa\in \R$, we get that minimizers of $\Gamma_\lambda=\Gamma_\lambda(\kappa)$, which now depends on $\kappa$ (we will generally omit this from the notation where it is clear from the context), solve
\begin{equation*}
L u+ \kappa u -\gamma f(u)=0.
\end{equation*}
In order for $\J_\kappa$ to be non-negative on $H^{s/2}$, we need $\kappa>-\inf_{\xi\in \R} m(\xi)$.
Letting
\begin{equation}
\label{eq: scale 1}
\beta^{-1}v=u,  \quad \beta^{p-1}=\gamma
\end{equation}
one gets that $v$ solves  \eqref{eq: main equation2} with wave-speed $c=\kappa$. We will denote by $G_\lambda(\kappa)$ the set of minimizers of $\Gamma_\lambda(\kappa)$. For equation \eqref{eq: main equationZ} one can consider $\J(=\J_1)$ again and let
\begin{equation}
\label{eq: scale 2}
\beta^{-1}v=u,  \quad \beta^{p-1}=\kappa\gamma
\end{equation}
for $\kappa>0$. Then $v$ will be a solution to \eqref{eq: main equationZ} with wave speed $c=\kappa$. This is equivalent to consider $\kappa\J$ instead of $\J_\kappa$ in \eqref{eq: Gamma}, which in turn is equivalent to scaling $\lambda$ by some factor. Thus every wave speed $c>0$ can be attained as (the reciprocal of) the Lagrange multiplier by varying $\lambda$.

As in \cite{zeng}, we will also consider inhomogeneous nonlinearities of the form $g(u)=u+f(u)$, where $f$ satisfies Assumption (B). For solitary-wave solutions of \eqref{eq: main equation}, the difference between homogeneous nonlinearities $f$ and inhomogeneous nonlinearities $g(u)=u+f(u)$ is trivial. A variational formulation in terms of conserved quantities is given by minimizing $\E-\Q$ in place of $\E$ in \eqref{eq: constrained var. prob. 1}, which clearly makes no difference for the existence of minimizers. And every element of $G_\lambda(\kappa)$ will be a solitary-wave solution with wave speed $c=\kappa+1$ upon scaling as in \eqref{eq: scale 1}. For \eqref{eq: main equation 2} it is more complicated. Equation \eqref{eq: main equation 2} in this case becomes
\begin{equation}
\label{eq: secondary eq}
u_t+u_x+f(u)_x+\left(L u\right)_t=0
\end{equation}
and solitary-wave solutions satisfy
\begin{equation}
\label{eq: secondary eq 2}
c L u+(c-1)u-f(u)=0.
\end{equation}
For $\kappa>0$ such that $1-1/\kappa>-\inf_{\xi\in \R} m(\xi)$, every element of $G_\lambda(1-1/\kappa)$ will be a solution to \eqref{eq: secondary eq 2} with wave speed $c=\kappa$ upon scaling as in \eqref{eq: scale 2}. The functional $\J_{1-1/\kappa}$ is, however, not a preserved quantity for \eqref{eq: secondary eq}, nor is the functional $\U$, and we are therefore not able to prove stability of the set of minimizers. We consider instead $\J(=\J_1)$, set
\begin{equation*}
\tilde{\U}(u)=\int_{\R}\left(\frac{u^2}{2}+F(u)\right)\dx
\end{equation*}
and look for minimizers of
\begin{equation*}
\tilde{\Gamma}_{\lambda}=\inf \lbrace \J(w) : w\in H^{s/2} \,\, \text{and} \,\, \tilde{\U}(w)=\lambda\rbrace.
\end{equation*}
We denote by $\tilde{G}_\lambda$ the set of minimizers of $\tilde{\Gamma}_{\lambda}$. By the Lagrange multiplier principle, any element of $\tilde{G}_\lambda$ will be a solution of \eqref{eq: secondary eq 2} with ${c=1/\gamma}$, where $\gamma$ is the Lagrange multiplier. Furthermore, the functionals $\J$ and $\tilde{\U}$ are preserved quantities for \eqref{eq: secondary eq} and we can therefore prove stability for the set of minimizers of $\tilde{\Gamma}_\lambda$ (cf. Theorem \ref{thm: stability} and Section \ref{stable regular}). Note that since $\tilde{\U}$ is inhomogeneous, the scaling arguments performed on the elements of $G_{\lambda}$ in order to choose the wave speed cannot be performed for the minimizers of $\tilde{\Gamma}_{\lambda}$; we will only get the wave speeds given by the Lagrange multiplier principle. Moreover, existence of minimizers of $\tilde{\Gamma}_\lambda$ can only be established for $\lambda>\lambda_0$ for some $\lambda_0\geq 0$ whose precise value is unknown. Equation \eqref{eq: secondary eq} can thus be said to have more in common with \eqref{eq: main equation} than \eqref{eq: main equation 2} in that fixing the wave speed comes at the cost of stability. The precise details of our main results on existence and stability of solitary-wave solutions are contained in Theorem \ref{thm: existence} and Theorem \ref{thm: stability} below.

\begin{theorem}[Existence of solitary-wave solutions]
\label{thm: existence}
Assume $L$ satisfies Assumption (A) and $f$ satisfies Assumption (B). Then:

\begin{itemize}
	\item[(i)] If $p\in(1,2s+1)$, there is a number $q_0\geq 0$ such that set $D_q$ of minimizers of $I_q$ is non-empty for any $q>q_0$, and every element of $D_q$ is a solution to \eqref{eq: main equation2} with the wave speed $c$ being the Lagrange multiplier in this constrained variational problem. If, in addition to (A), $m(\xi)$ satisfies $0\leq m(\xi)\leq A_2 |\xi|^s$ for $|\xi|\leq 1$, then $q_0=0$.
	
	\item[(ii)] If $p\in (1,\frac{1+s}{1-s})$, the set $G_\lambda=G_\lambda(\kappa)$ of minimizers of $\Gamma_\lambda=\Gamma_\lambda(\kappa)$ is non-empty for any $\lambda>0,\kappa>-\inf_{\xi\in \R} m(\xi)$, and if $f$ satisfies (B2), then this is true also for $\lambda<0$. If $\kappa=1$, then every element of $G_\lambda$ solves \eqref{eq: main equationZ} with the wave speed $c$ being the reciprocal of the Lagrange multiplier in this constrained variational problem, and by varying the parameter $\lambda$ one can get any wave speed $c>0$. Moreover, scaling the set $G_\lambda(\kappa)$ as in \eqref{eq: scale 1}, or the set $G_\lambda(1-1/\kappa)$ (for $k>0$ such that $1-1/\kappa>-\inf_{\xi\in \R}m(\xi)$) as in \eqref{eq: scale 2}, every element will be a solution to \eqref{eq: main equation2} or \eqref{eq: secondary eq 2}, respectively, with wave speed $c=\kappa$.

	\item[(iii)] If $p\in (1,\frac{1+s}{1-s})$, there exists a $\lambda_0\geq 0$ such that the set $\tilde{G}_\lambda$ of minimizers of $\tilde{\Gamma}_\lambda$ is non-empty for any $\lambda>\lambda_0$, and every element of $\tilde{G}_\lambda$ is a solution to \eqref{eq: secondary eq 2} with the wave speed $c$ being the reciprocal of the Lagrange multiplier in this constrained variational problem. If, in addition to (A), $m(\xi)$ satisfies $0\leq m(\xi)\leq A_2 |\xi|^s$ for $|\xi|\leq 1$, then $\lambda_0=0$ for $p\in(1,2s+1)$.
\end{itemize}
Moreover, if $\lbrace u_n\rbrace_n\subset H^{s/2}$ is a minimizing sequence of $I_q$,  $\Gamma_\lambda(\kappa)$ or $\tilde{\Gamma}_\lambda$, under the conditions of (i), (ii) or (iii), respectively, then there exists a sequence $\lbrace y_n\rbrace\subset \R$ such that a subsequence of $\lbrace u_n(\cdot +y_n)\rbrace_n$ converges in $H^{s/2}$ to an element of $D_q$, $G_\lambda(\kappa)$ or $\tilde{G}_\lambda$, respectively. Furthermore, $D_q,G_\lambda(\kappa),\tilde{G}_\lambda\subset H^{s}$.
\end{theorem}
The theorem above covers for instance the Capillary Whitham equation, for which solitary waves moving to the right with speed $c>0$ are described by the equation
\begin{equation}
\label{eq: capillary whitham}
	Lu-cu+u^2=0,
\end{equation}
where $m(\xi)=\sqrt{\frac{(1+\beta \xi^2)\tanh(\xi)}{\xi}}$, with $\beta>0$ being the strength of the surface tension. For any $\beta>0$, $m$ satisfies (A) with $s=1/2$, and $\inf_{\xi\in \R} m(\xi)=\min_{\xi \in \R}m(\xi)>0$ and $f(u)=-u^2$ satisfies (B2). Hence the set $G_\lambda(\kappa)$ is non-empty for any $\kappa>-\inf_{\xi\in \R} m(\xi)$, and in particular, for $-\inf_{\xi\in \R} m(\xi)<\kappa<0$, we get solutions to \eqref{eq: capillary whitham} with wave-speed $c=-\kappa>0$.

\begin{remark}
\label{rem: remark on scaling}
Scaling elements of $G_\lambda(\kappa)$ in order to choose the wave speed comes at the cost of losing information about the quantity $\U(u)$ for solutions $u$. For given energy $\U(u)=\lambda$ one faces the opposite problem, that the wave speed $c$ is given as the reciprocal of the Lagrange multiplier which one cannot directly control. However, the Lagrange multiplier $\gamma$, which is the unknown factor in the scalings \eqref{eq: scale 1} and \eqref{eq: scale 2}, can be expressed in terms of the quantities $\lambda$, $p$ and $\Gamma_\lambda$ as follows (see Section \ref{existence}):
\begin{equation}
\label{eq: lagrange multiplier}
\gamma=\frac{2\Gamma_\lambda}{(p+1)\lambda}.
\end{equation}
This expression illustrates at least the relationship between the different quantities $c$, $\lambda$ and $\Gamma_\lambda$.
\end{remark}

\begin{theorem}[Conditional energetic stability]
\label{thm: stability}
The sets $D_q$, (any positive scaling of) $G_\lambda(1)$ with $\kappa=1$ and $\tilde{G}_\lambda$ are, under the conditions in Theorem \ref{thm: existence} (i), (ii) and (iii) respectively, stable sets for the initial value problems of \eqref{eq: main equation}, \eqref{eq: main equation 2} and \eqref{eq: secondary eq}, respectively, in the following sense as described for $D_q$: For every $\varepsilon>0$ there exists $\delta>0$ such that if
\begin{equation*}
\inf_{w\in D_q}\|u_0-w\|_{H^{s/2}}<\delta,
\end{equation*}
where $u(x,t)$ solves \eqref{eq: main equation} with $u(x,0)=u_0(x)$, then
\begin{equation*}
\inf_{w\in D_q}\|u(\cdot,t)-w\|_{H^{s/2}}<\varepsilon
\end{equation*}
for all $t\in \R$.
\end{theorem}

\begin{remark}
While the upper bounds $2s+1$ and $\frac{1+s}{1-s}$ on $p$ appear in the proof of Theorem \ref{thm: existence} by appealing to Sobolev embedding and interpolation theorems rather than from the equations themselves, they are, in fact, strictly related to existence and stability of solitary-wave solutions. In \cite{Linares2014dpo}, it is proven that for $m(\xi)=|\xi|^s$ and $p=2$, there are no non-trivial solutions to \eqref{eq: main equation2} if $s<1/3$. If $s=1/3$, then $p=2$  is the upper bound $\frac{1+s}{1-s}$. Their arguments can easily be generalized to show that if $p>\frac{1+s}{1-s}$, for any $s>0$, there are no solutions to \eqref{eq: main equation2}. Similarly, as already mentioned one has instability for $p>2s+1$ for equations like the GKdV, and this limitation on $p$ is therefore also not due to any limitations of the proofs presented in this paper.
\end{remark}

The following table summarizes the essential content of Theorems \ref{thm: existence} and \ref{thm: stability} in terms of which nonlinearities one has existence for, and for which one has stability, for equations \eqref{eq: main equation}, \eqref{eq: main equation 2} and \eqref{eq: secondary eq} (here L. multiplier is short for Lagrange multiplier).
\begin{table}[h]
\begin{tabular}{|l|l|l|l|}
\hline
Equation & Wave speed & Existence & Stability \\ \hline
\multirow{2}{*}{\eqref{eq: main equation}} & any $c>0$ & $p\in (1,\frac{1+s}{1-s})$ & \\ \cline{2-4}
& L. multiplier & $p\in (1,2s+1)$ & $p\in (1,2s+1)$ \\ \hline
\eqref{eq: main equation 2} & any $c>0$/ & $p\in (1, \frac{1+s}{1-s})$ & $p\in (1, \frac{1+s}{1-s})$ \\ 
 & L. multiplier & & \\ \hline
\multirow{2}{*}{\eqref{eq: secondary eq}} & any $c>1$ & $p\in (1, \frac{1+s}{1-s})$ & \\ \cline{2-4}
& L. multiplier & $p\in (1, \frac{1+s}{1-s})$ & $p\in (1, \frac{1+s}{1-s})$ \\
\hline
\end{tabular}
\caption{Ranges of existence and stability of solitary-wave solutions of \eqref{eq: main equation}, \eqref{eq: main equation 2} and \eqref{eq: secondary eq} in terms of the exponent $p$ of the nonlinearity. The Lagrange multipliers come from variational formulations in terms of conserved quantities, while "any" $c$ is obtained through scaling arguments.} 
\label{table 1}
\end{table}

\begin{remark}
Recall that existence and stability of solitary-waves for equation $u_t+u_x+\left(f(u)\right)_x-\left(L u\right)_x=0$ is equivalent to that of $u_t+\left(f(u)\right)_x-\left(L u\right)_x=0$ (see the discussion leading up to Theorem \ref{thm: existence}).
\end{remark}

We end the section by stating the concentration-compactness lemma that will be the main ingredient in the sequel:

\begin{lemma}[Lions \cite{lions}]
\label{Lions}
Let $\lbrace \rho_n\rbrace_n\subset L^1$ be a sequence that satisfies
\begin{align*}
\rho_n & \geq 0 \,\, \text{a.e. on} \,\, \R,\\
\int_\R \rho_n\dx &=\mu
\end{align*}
for a fixed $\mu>0$ and all $n\in \N$. Then there exists a subsequence $\lbrace \rho_{n_k}\rbrace_k$ that satisfies one of the three following properties:
\begin{itemize}
\item[(1)] (Compactness). There exists a sequence $\lbrace y_k\rbrace_k\subset \R$ such that for every $\varepsilon>0$, there exists $r<\infty$ satisfying for all $k\in \N$:
\begin{equation*}
\int_{y_k-r}^{y_k+r}\rho_{n_k}(x)\dx\geq \mu-\varepsilon.
\end{equation*}
\item[(2)] (Vanishing). For all $r<\infty$,
\begin{equation*}
\lim_{k\rightarrow\infty}\sup_{y\in \R}\int_{y-r}^{y+r}\rho_{n_k}\dx=0
\end{equation*}
\item[(3)] (Dichotomy). There exists $\bar{\mu}\in (0,\mu)$ such that for every $\varepsilon>0$ there exists a natural number $k_0\geq 1$ and two sequences of positive $L^1$ functions $\lbrace \rho_k^{(1)}\rbrace_k, \lbrace \rho_k^{(2)}\rbrace_k$ satisfying for $k\geq k_0$,
\begin{align}
\label{eq: dichotomy}
&\|\rho_{n_k}-(\rho_k^{(1)}+\rho_k^{(2)})\|_{L^1}\leq \varepsilon, \nonumber \\
&|\int_\R \rho_k^{(1)}\dx-\bar{\mu}|\leq \varepsilon, \\
&|\int_\R \rho_k^{(2)}\dx-(\mu-\bar{\mu})|\leq \varepsilon, \nonumber \\
& \mathrm{dist}(\mathrm{supp}(\rho_k^{(1)}),\mathrm{supp}(\rho_k^{(2)}))\rightarrow \infty. \nonumber
\end{align}
\end{itemize}
\end{lemma}

\begin{remark}
\label{rem: Lions}
The condition $\int_{\R}\rho_n\dx=\mu$ can be replaced by $\int_{\R}\rho_n\dx=\mu_n$ where $\mu_n\rightarrow\mu$ (see \cite{chen}).
\end{remark}
\section{Concentration-compactness for \eqref{eq: Gamma}}
\label{existence}
The variational problem
\begin{equation}
\Gamma_\lambda= \inf \lbrace \J_\kappa(w) : w\in H^{s/2} \,\, \text{and} \,\, \U(w)=\lambda\rbrace.
\end{equation}
is equivalent to the one considered in \cite{weinstein}, where it was arrived at by first considering the functional
\begin{equation*}
J(u) =\frac{\frac{1}{2}\int_{\R}\left( uL u+\kappa u^2\right)\dx}{\left(\int_{\R}F(u)\dx\right)^{\frac{2}{p+1}}},
\end{equation*}
for some constant $\kappa >0$ and noting that it is invariant under the scaling $u\mapsto \theta u$ for $\theta\neq 0$. As minimizers of the constrained variational problem then also minimize the unconstrained functional over $H^{s/2}$, one can ascertain some a-priori information about the sign and size of the wave speed $1/\gamma$ in terms of the quantities $p$, $\lambda$ and $\Gamma_\lambda$. We henceforth assume $p\in (1,\frac{1+s}{1-s})$, so that by the Sobolev embedding theorem, $\int_\R F(u)\dx$ is finite for all $u\in H^{s/2}$.

Assume now that $u$ is a minimizer of $\Gamma_\lambda$. Then
\begin{equation*}
\frac{\mathrm{d}}{\mathrm{d}t} J(u+t\varphi)\vert_{t=0}=0,
\end{equation*}
for all $\varphi\in H^{s/2}$. Calculating the derivative, we get
\begin{align*}
\int_{\R} & \varphi L u + \kappa \varphi u\dx \left(\int_{\R} F(u)\dx\right)^{\frac{2}{p+1}} \\
&-\frac{1}{p+1}\int_{\R} uL u+\kappa u^2\dx\left( \int_{\R} F(u)\dx\right)^{\frac{1-p}{p+1}} \int_{\R}\varphi f(u)\dx=0.
\end{align*}
That is,
\begin{equation*}
L u+\kappa u-\frac{1}{p+1}\int_{\R} uL u+\kappa u^2\dx \left(\int_{\R} F(u)\dx\right)^{-1} f(u)=0.
\end{equation*}
Thus, if $\kappa =1$ and $\lambda>0$, minimizers of $\Gamma_\lambda$ will be solutions of equation \eqref{eq: main equationZ} with wave speed $\frac{(p+1)\lambda}{2 \Gamma_\lambda}>0$. Moreover, this establishes the expression for $\gamma$ given in Remark \ref{rem: remark on scaling}.

Now we turn to the existence of minimizers of \eqref{eq: Gamma}, which we will prove using concentration-compactness arguments. As $u\in H^{s/2}$ implies $F(u)\in L^1$ and we fix $\U(u)=\lambda$, it would be natural for a minimizing sequence $\lbrace u_n\rbrace_n$ of $\Gamma_\lambda$ to apply Lemma \ref{Lions} to $\lbrace F(u_n)\rbrace_n$ as in \cite{weinstein}. Unfortunately, $F(u)$ does not satisfy the non-negativity criterion for all nonlinearities $f$ we would like to consider. Nor does the other natural candidate $uL u+\kappa u^2$. We therefore replace $u L u$ with a non-negative term the integral of which (over $\R$) is equal to that of $u L u$. We define the operator $L^{\frac{1}{2}}$ by replacing $m$ with $\sqrt{m}$ in Assumption (A), and let
\begin{align*}
\rho_n&=\kappa u_n^2+\left(L^\frac{1}{2}u_n\right)^2, \\ \mu_n&=\int_{\R}\rho_n\dx.
\end{align*}
Thus $\rho_n\geq0$ and there exist $k_1, k_2>0$ depending on $\kappa$ such that
\begin{equation}
\label{eq: upper and lower bound on mu}
k_1\|u_n\|_{H^{s/2}}^2\leq \int_{\R}\rho_n\dx\leq k_2\|u_n\|_{H^{s/2}}^2.
\end{equation}
In order to apply Lemma \ref{Lions} we will need the following lemma.
\begin{lemma}
\label{lem: lower bound on Hs norm}
If $\lbrace u_n\rbrace_n$ is a minimizing sequence of $\Gamma_\lambda$, then there exists $M>0$ and $N>0$ such that $N\leq \|u_n\|_{H^{s/2}}\leq M$ for all $n$. Furthermore, $\Gamma_\lambda>0$.
\end{lemma}
\begin{proof}
Noting that Assumption (A) implies that for any $\kappa>-\inf_{\xi\in \R} m(\xi)$, $\left(\J_\kappa(\cdot)\right)^{1/2}$ defines a norm on $H^{s/2}$ equivalent to the standard norm, the upper bound follows trivially from the boundedness of $\lbrace \J_\kappa(u_n)\rbrace_n\subset \R$. Similarly, the lower bound is a consequence of $\int_{\R} |F(u)|\dx= \lambda$ and the Sobolev embedding theorem. That $\Gamma_\lambda>0$ is an immediate consequence of the lower bound.
\end{proof}
By \eqref{eq: upper and lower bound on mu} and Lemma \ref{lem: lower bound on Hs norm}, for any minimizing sequence $\lbrace u_n\rbrace_n\subset H^{s/2}$ of $\Gamma_\lambda$, the sequence $\lbrace \mu_n\rbrace_n\subset \R$ as defined above will be bounded. Moreover, $\mu_n>0$ for all $n$. Thus there exists a number $\mu>0$ and a subsequence of $\lbrace \rho_n\rbrace_n$, still denoted by $\lbrace\rho_n\rbrace_n$, such that ${\int_\R \rho_n\rightarrow \mu}$. By Remark \ref{rem: Lions}, Lemma \ref{Lions} then applies to the sequence $\lbrace \rho_n\rbrace_n$ and there exists a subsequence, still denoted by $\lbrace \rho_n\rbrace_n$, for which either compactness, vanishing or dichotomy holds. In what follows we will eliminate vanishing and dichotomy. To this purpose, we will first establish some structural properties of $\Gamma_\lambda$ considered as a function of $\lambda$, as well as some general properties of minimizing sequences for $\Gamma_\lambda$.

We start with the following Lemma from \cite{zeng} (Lemma 2.9).
\begin{lemma}
\label{lem: Zeng 3}
If $\lambda_2>\lambda_1>0$, then $\Gamma_{\lambda_2}\geq \Gamma_{\lambda_1}$.
\end{lemma}
\begin{proof}
For any $\varepsilon>0$, there exists a function $\varphi\in H^{s/2}$ such that $\U(\varphi)=\lambda_2$ and $\J_{\kappa}(\varphi)\leq \Gamma_{\lambda_2}+\varepsilon$. Since $\U(a\varphi)$ is a continuous function of $a\in \R$, then by the intermediate value theorem we can find $C\in (0,1)$ such that $\U(C\varphi)=\lambda_1$. Hence
\begin{equation*}
\Gamma_{\lambda_1}\leq \J_{\kappa}(C\varphi)=C^2\J_{\kappa}(\varphi)< \J_{\kappa}(\varphi)<\Gamma_{\lambda_2}+\varepsilon.
\end{equation*}
This proves the result.
\end{proof}

\begin{lemma}
\label{lem: Zeng 4}
For $\lambda>0$ and any $\alpha\in (0,\lambda)$,
\begin{equation*}
\Gamma_\lambda<\Gamma_{\lambda-\alpha}+\Gamma_\alpha.
\end{equation*}
\end{lemma}
\begin{proof}
Let $\theta\in (1, \lambda \alpha^{-1})$. Then,
\begin{align*}
\Gamma_{\theta\alpha} &=\inf \lbrace \J_{\kappa}(u) : u\in H^{s/2}, \,\, \int_{\R}F(u)\dx=\alpha\theta\rbrace \nonumber \\
&=\inf \lbrace \J_{\kappa}(\theta^\frac{1}{p+1}v) : v\in H^{s/2}, \,\,\int_{\R}F(v)\dx=\alpha \rbrace \nonumber \\
&=\theta^\frac{2}{p+1}\Gamma_\alpha \nonumber \\
&<\theta \Gamma_\alpha,
\end{align*}
where the last inequality follows from that $\theta>1$ and, by Assumption (B), $p>1$.
Now if $\alpha\geq \lambda-\alpha$,
\begin{align}
\label{eq: Gamma_a<Gamma_b+Gamma_(a-b)}
\Gamma_{\lambda} &=\Gamma_{\lambda-\alpha+\alpha}=\Gamma_{\alpha\left(1+ \frac{\lambda-\alpha}{\alpha}\right)}<\left(1+ \frac{\lambda-\alpha}{\alpha}\right)\Gamma_\alpha \nonumber \\
&=\Gamma_\alpha+\frac{\lambda-\alpha}{\alpha}\Gamma_{\frac{\alpha}{\lambda-\alpha}(\lambda-\alpha)}<\Gamma_\alpha+\Gamma_{\lambda-\alpha}.
\end{align}
For $\alpha\leq \lambda-\alpha$ one can derive the same inequality in a similar manner.
\end{proof}

To exclude vanishing, we will need the following result from \cite{albert}:
\begin{lemma}
\label{lem: lower bound on p+1 norm}
Given $K>0$ and $\delta>0$, there exists $\eta=\eta(K,\delta)>0$ such that if $v\in H^{s/2}$ with $\|v\|_{H^{s/2}}\leq K$ and $\|v\|_{L^{p+1}}\geq \delta$, then
\begin{equation*}
\sup_{y\in\R}\int_{y-2}^{y+2}|v(x)|^{p+1}\dx\geq \eta.
\end{equation*}
\end{lemma}
\begin{proof}
Without loss of generality we may assume $s/2\leq 1$; if $s/2>1$ then $\|v\|_{H^1}\leq K$ and the argumentation that follows can be carried out for $H^1$. Choose a smooth function $\zeta: \R\rightarrow [0,1]$ with support in $[-2,2]$ and satisfying ${\sum_{j\in\Z}\zeta(x-j)=1}$ for all $x\in \R$, and define $\zeta_j(x)=\zeta(x-j)$ for $j\in \Z$. The map $T: H^r\rightarrow l_2(H^r)$ defined by
\begin{equation*}
Tv=\lbrace\zeta_jv\rbrace_{j\in\Z}
\end{equation*}
is easily seen to be bounded for $r=0$ and $r=1$. For $r=0$, 
\begin{equation*}
\|Tv\|_{l_2(L^2)}^2=\sum_{j\in \Z} \|\zeta_j v\|_{L^2}^2\leq \sum_{j\in \Z} \int_{-2+j}^{2+j}v^2\dx=4\|v\|_{L^2}^2,
\end{equation*}
and one can argue similarly when $r=1$, recalling that $\zeta$ is a smooth function. By interpolation the map $T$ is therefore also bounded for $r=s/2$. That is, there exists a constant $C_0$ such that for all $v\in H^{s/2}$,
\begin{equation*}
\sum_{j\in \Z}\|\zeta_jv\|_{H^{s/2}}^2\leq C_0\|v\|_{H^{s/2}}^2.
\end{equation*}
Since $l_{p+1}\hookrightarrow l_1$, there exists a positive number $C_1$ such that $\sum_{j\in\Z}|\zeta(x-j)|^{p+1}\geq C_1$ for all $x\in \R$. We claim that for every $v\in H^{s/2}$ that is not identically zero, there exist an integer $j_0$ such that
\begin{equation}
\label{eq: lower bound on p+1 norm}
\|\zeta_{j_0}v\|_{H^{s/2}}^2\leq \left(1+C_2|\|v\|_{L^{p+1}}^{-p-1}\right)\|\zeta_{j_0}v\|_{L^{p+1}}^{p+1},
\end{equation}
where $C_2=C_0K^2/C_1$. To see this, assume to the contrary that
\begin{equation*}
\|\zeta_{j}v\|_{H^{s/2}}^2> \left(1+C_2|\|v\|_{L^{p+1}}^{-p-1}\right)\|\zeta_{j}v\|_{L^{p+1}}^{p+1},
\end{equation*}
for every $j\in \Z$. Summing over $j$, we obtain
\begin{equation*}
C_0\|v\|_{H^{s/2}}^2>\left(1+C_2|\|v\|_{L^{p+1}}^{-p-1}\right)\sum_{j\in\Z}\|\zeta_{j}v\|_{L^{p+1}}^{p+1}
\end{equation*}
and hence by our choice of $C_2$
\begin{equation*}
C_0K^2>\left(1+C_2|\|v\|_{L^{p+1}}^{-p-1}\right)C_1\|v\|_{L^{p+1}}^{p+1}=C_1\|v\|_{L^{p+1}}^{p+1}+C_0K^2,
\end{equation*}
for $\|v\|_{H^{s/2}}\leq K$, which is a contradiction. This proves \eqref{eq: lower bound on p+1 norm}.

Observe now that from \eqref{eq: lower bound on p+1 norm} and the assumptions of the lemma it follows that
\begin{equation*}
\|\zeta_{j_0}v\|_{H^{s/2}}^2\leq \left(1+C_2/\delta^{p+1}\right)\|\zeta_{j_0}v\|_{L^{p+1}}^{p+1}.
\end{equation*}

For $p\leq \frac{1+s}{1-s}$, we have by the Sobolev embedding theorem that
\begin{equation*}
\|\zeta_{j_0}v\|_{L^{p+1}}\leq C\|\zeta_{j_0}v\|_{H^{s/2}}
\end{equation*}
where $C$ is independent of $v$. Combining the above two inequalities we get that
\begin{equation*}
\|\zeta_{j_0}v\|_{L^{p+1}}\geq \left[C^2\left(1+C_2/\delta^3\right)\right]^{1/(1-p)},
\end{equation*}
and since
\begin{equation*}
\int_{j_0-2}^{j_0+2}|v|^{p+1}\dx\geq \|\zeta_{j_0}v\|_{L^{p+1}}^{p+1}
\end{equation*}
the result follows, with $\eta=\left[C^2\left(1+C_2/\delta^3\right)\right]^{(p+1)/(1-p)}$.
\end{proof}
We may now exclude vanishing.
\begin{lemma}
\label{lem: no vanishing Z}
Vanishing does not occur.
\end{lemma}
\begin{proof}
Let $\lbrace u_n\rbrace_n$ be a minimizing sequence for $\Gamma_\lambda$. By Assumption (B), the constraint $\int_{\R} F(u_n)\dx=\lambda>0$ for all $n\in \N$ implies that $\|u_n\|_{L^{p+1}}\geq \delta$ for $\delta=(|c_p^{-1}|\lambda)^{1/(p+1)}$. By Lemma \ref{lem: lower bound on Hs norm}, we additionally have that there is a $K$ such that $\|u_n\|_{H^{s/2}}\leq K$ for all $n\in \N$. The criteria for Lemma \ref{lem: lower bound on p+1 norm} are therefore satisfied for all $n\in \N$ and there exists an $\eta(K, \delta)>0$ such that $\sup_{y\in \R} \int_{y-2}^{y+2}|u_n|^{p+1}\dx\geq \eta$ for all $n\in \N$. The result now follows from the Sobolev embedding theorem:
\begin{equation*}
C \eta^{2/p+1}\leq \sup_{y\in \R} C\left(\int_{y-2}^{y+2} |u_n|^{p+1}\dx\right)^{2/p+1}\leq \sup_{y\in \R} \int_{y-2}^{y+2} \rho_n \dx,
\end{equation*}
for all $n\in \N$, where $C>0$ is an embedding constant.
\end{proof}

Now it only remains to preclude dichotomy. The following theorem is the key result in order to do so.
\begin{theorem}
\label{thm: dichotomy convergence}
Assume that $L$ satisfies Assumption (A). Let $u\in H^{s/2}$ and $\varphi, \psi\in C^\infty$ satisfy $0\leq \varphi\leq 1$, $0\leq \psi\leq 1$,
\begin{equation*}
\varphi(x)=\left\{
	\begin{array}{l l}
		1, & \quad \text{if}\,\, |x|<1,\\
		0, & \quad \text{if}\,\, |x|>2,
	\end{array} \right.
\end{equation*} 
and
\begin{equation*}
\psi(x)=\left\{
	\begin{array}{l l}
		0, & \quad \text{if}\,\, |x|<1,\\
		1, & \quad \text{if}\,\, |x|>2.
	\end{array} \right.
\end{equation*}
Define $\varphi_r(x)=\varphi(x/r)$ and $\psi_r(x)=\psi(x/r)$ for all $x\in \R$. Then for all $r>0$ sufficiently large,
\begin{equation*}
\left|\int_{\R}\varphi_r u(L(\varphi_ru)-\varphi_rL u)\dx\right|\leq \beta(r) \|u\|_{H^{s/2}}^2
\end{equation*}
and
\begin{equation*}
\left|\int_{\R}\psi_r u(L(\psi_ru)-\psi_rL u)\dx\right|\leq \beta(r) \|u\|_{H^{s/2}}^2,
\end{equation*}
where $\beta(r)\rightarrow 0$ as $r\rightarrow \infty$. In particular, the integrals above converge to $0$ as $r\rightarrow \infty$ uniformly in $u\in H^{s/2}$.
\end{theorem}
\begin{proof}
By Plancherel's theorem, basic properties of the Fourier transform, and Fubini's theorem,
\begin{align}
\int_{\R}\varphi_r u(L & (\varphi_ru)-\varphi_rL u)\dx \nonumber \\
& =\int_{\R}\overline{\widehat{\varphi_r u}}(\xi)\left[m(\xi)(\widehat{\varphi_r}\ast \widehat{u})(\xi)-(\widehat{\varphi_r}\ast (m\widehat{u}))(\xi)\right]\dxi \nonumber \nonumber \\
& =\int_{\R} \overline{\widehat{\varphi_r u}}(\xi)\int_{\R}\widehat{\varphi_r}(t)\widehat{u}(\xi-t)(m(\xi)-m(\xi-t))\, \mathrm{d}t\dxi \nonumber\\
& =\int_\R \int_\R \overline{\widehat{\varphi_r u}}(\xi)\widehat{\varphi_r}(t)\widehat{u}(\xi-t)(m(\xi)-m(\xi-t))\, \mathrm{d}t\dxi \nonumber\\
& =\int_\R \widehat{\varphi_r}(t)\int_\R \overline{\widehat{\varphi_r u}}(\xi)\widehat{u}(\xi-t)(m(\xi)-m(\xi-t))\dxi\, \mathrm{d}t. \label{eq: rewriting on fourier side}
\end{align}
We can write $m(\xi-t)=(1+\sqrt{m(\xi)})\sqrt{m(\xi-t)}\frac{\sqrt{m(\xi-t)}}{1+\sqrt{m(\xi)}}$. By assumption,
\begin{equation*}
\frac{\sqrt{m(\xi-t)}}{1+\sqrt{m(\xi)}}\leq \frac{A_2^{1/2}|\xi-t|^{s/2}}{1+A_1^{1/2}|\xi|^{s/2}}\leq C(1+|t|^{s/2})
\end{equation*}
for some $C$ independent of $t$ and $\xi$. Hence, by H\"older's inequality,
\begin{align*}
\Big|\int_\R &\overline{\widehat{\varphi_r u}}(\xi)\widehat{u}(\xi-t)m(\xi-t)\dxi\Big| \\
= & \left|\int_\R \overline{\widehat{\varphi_r u}}(\xi)(1+\sqrt{m(\xi)})\widehat{u}(\xi-t)\sqrt{m(\xi-t)}\frac{\sqrt{m(\xi-t)}}{1+\sqrt{m(\xi)}}\dxi\right| \\
\leq & C(1+|t|^{s/2})\|\overline{\widehat{\varphi_r u}}(1+\sqrt{m})\|_{L^2}\|\widehat{u}\sqrt{m}\|_{L^2} \\
\leq & C_1(1+|t|^{s/2})\|u\|_{H^{s/2}}^2,
\end{align*}
where $C_1$ is independent of $t$ and $r$. Arguing in the same way, we find that
\begin{equation*}
\left|\int_\R \overline{\widehat{\varphi_r u}}(\xi)\widehat{u}(\xi-t)m(\xi)\dxi\right|\leq C_2 (1+|t|^{s/2})\|u\|_{H^{s/2}}^2,
\end{equation*}
where $C_2$ is independent of $t$ and $r$. Hence 
\begin{equation*}
\left|\int_\R \overline{\widehat{\varphi_r u}}(\xi)\widehat{u}(\xi-t)(m(\xi)-m(\xi-t))\dxi\right|\leq C (1+|t|^{s/2}) \|u\|_{H^{s/2}}^2
\end{equation*}
for some $C>0$ independently of $r$. For any $\alpha<1$, we have that
\begin{align*}
\Big|\int_{|t|>r^{-\alpha}} \widehat{\varphi_r}(t)\int_\R \overline{\widehat{\varphi_r u}}(\xi) & \widehat{u}(\xi-t)(m(\xi)-m(\xi-t))\dxi\, \mathrm{d}t\Big| \\
& \leq C\|u\|_{H^{s/2}}^2\int_{|t|>r^{-\alpha}}|\widehat{\varphi_r}(t)|(1+|t|^{s/2})\, \mathrm{d}t
\end{align*}
As $\varphi$ is a Schwartz function $\widehat{\varphi_r}$ approximates unity as $r\rightarrow \infty$ and
\begin{equation*}
\int_{|t|>r^{-\alpha}}|\widehat{\varphi_r}(t)|(1+|t|^{s/2})\, \mathrm{d}t\rightarrow 0
\end{equation*}
as $r\rightarrow \infty$ for any $\alpha<1$. It remains to consider
\begin{equation*}
\int_{|t|<r^{-\alpha}} \widehat{\varphi_r}(t)\int_\R \overline{\widehat{\varphi_r u}}(\xi)\widehat{u}(\xi-t)(m(\xi)-m(\xi-t))\dxi\, \mathrm{d}t.
\end{equation*}
Let $\varepsilon>0$ be given, and let us first assume that $m(\xi)$ has no discontinuities. Then $m(\xi)$ is uniformly continuous on any bounded domain and hence there exists a number $R=R(r)$ such that $\lim_{r\rightarrow \infty} R(r)=\infty$ and $|m(\xi)-m(\xi-t)|<\varepsilon$ for all $|\xi|\leq R$ and $|t|<r^{-\alpha}$. Thus
\begin{align}
&\int_{|t|<r^{-\alpha}} \widehat{\varphi_r}(t) \left|\int_\R \overline{\widehat{\varphi_r u}}(\xi)\widehat{u}(\xi-t)(m(\xi)-m(\xi-t))\dxi\right|\, \mathrm{d}t \leq \varepsilon C\|u\|_{L^2}^2 \nonumber \\
& +\int_{|t|<r^{-\alpha}} \widehat{\varphi_r}(t) \left|\int_{|\xi|>R} \overline{\widehat{\varphi_r u}}(\xi)\widehat{u}(\xi-t)(m(\xi)-m(\xi-t))\dxi\right|\, \mathrm{d}t. \label{eq: bound when t<r}
\end{align}
Note that $\varepsilon$ can be made arbitrarily small by taking $r$ sufficiently large independently of $\|u\|_{H^{s/2}}$. Hence it remains only to show that the second term also converges to $0$ uniformly in $u\in H^{s/2}$. Assumption \eqref{assumption on m 2} implies that
\begin{align}
\int_{|t|<r^{-\alpha}} & \widehat{\varphi_r}(t) \left|\int_{|\xi|>R} \overline{\widehat{\varphi_r u}}(\xi)\widehat{u}(\xi-t)(m(\xi)-m(\xi-t))\dxi\right|\, \mathrm{d}t \nonumber \\
& \leq \int_{|t|<r^{-\alpha}} \widehat{\varphi_r}(t)\int_{|\xi|>R} |\overline{\widehat{\varphi_r u}}(\xi)\widehat{u}(\xi-t)||m(\xi)-m(\xi-t)|\dxi\, \mathrm{d}t \nonumber \\
& \leq \int_{|t|<r^{-\alpha}} \widehat{\varphi_r}(t) k(t)\int_{|\xi|>R} |\overline{\widehat{\varphi_r u}}(\xi)\widehat{u}(\xi-t)||\xi|^s\dxi\, \mathrm{d}t \nonumber \\
& \leq \|u\|_{H^{s/2}}^2\int_{|t|<r^{-\alpha}} \widehat{\varphi_r}(t) k(t)\, \mathrm{d}t \nonumber \\
& \leq \left(\sup_{|t|<r^{-\alpha}}k(t)\right) \|\widehat{\varphi}\|_{L^1} \|u\|_{H^{s/2}}^2. \label{eq: bound t something}
\end{align}
By assumption, $\lim_{r\rightarrow \infty} \sup_{|t|<r^{-\alpha}}k(t)=0$. This proves the first part when $m$ is continuous. Now let $m$ have a finite number of discontinuities. The inequalities \eqref{eq: bound when t<r} and \eqref{eq: bound t something} fail in an interval of length $2r^{-\alpha}$ around each discontinuity. As the number of discontinuities is finite, the total measure of the set where the inequalities fail therefore goes to $0$ as $r\rightarrow \infty$. Hence the result holds also in this case.

To prove the result for $\left|\int_{\R}\psi_r u(L(\psi_ru)-\psi_rL u)\dx\right|$, note that without loss of generality it can be assumed that $\psi_r=1-\varphi_r$. Then
\begin{align*}
\int_{\R}\psi_r u(L(\psi_ru)-\psi_rL u)\dx & =\int_{\R}\varphi_r u(L(\varphi_ru)-\varphi_rL u)\dx \\
& -\int_{\R}u(L(\varphi_ru)-\varphi_rL u)\dx.
\end{align*}
The first integral on the right-hand side is exactly what we had above, while the second integral can be written as in \eqref{eq: rewriting on fourier side} with $\overline{\widehat{\varphi_r u}}(\xi)$ replaced by $\overline{\widehat{u}}(\xi)$, which does not change the estimates.
\end{proof}
Using Theorem \ref{thm: dichotomy convergence} we are able to prove the equivalent of Lemma 2.15 in \cite{zeng} for a much larger class of operators $L$, in particular extending the result to operators of order $0<s<1$.
\begin{lemma}
\label{lem: dichotomy Z}
Assume the dichotomy alternative holds for $\rho_n$. Then for each $\varepsilon>0$ there is a subsequence of $\lbrace u_n\rbrace_n$, still denoted $\lbrace u_n\rbrace_n$, a real number $\bar{\lambda}=\bar{\lambda}(\varepsilon)$, $N\in \N$ and two sequences $\lbrace u_n^{(1)}\rbrace_n, \lbrace u_n^{(2)}\rbrace_n\subset H^{s/2}$ satisfying for all $n\geq N$:
\begin{subequations}
\begin{align}
&|\U(u_n^{(1)})-\bar{\lambda}|\leq\varepsilon, \label{eq: a Z} \\
& |\U(u_n^{(2)})-(\lambda-\bar{\lambda})|<\varepsilon, \label{eq: b Z} \\
&|\J_{\kappa}(u_n)-\J_{\kappa}(u_n^{(1)})-\J_{\kappa}(u_n^{(2)})|<\varepsilon \label{eq: c Z}.
\end{align}
\end{subequations}
Furthermore,
\begin{equation}
\label{eq: G(u^1)=mu}
| \J_{\kappa}(u_n^{(1)})-\bar{\mu}| \leq\varepsilon
\end{equation}
and
\begin{equation}
\label{eq: G(u^2)}
| \J_{\kappa}(u_n^{(2)})-(\mu-\bar{\mu})|\leq \varepsilon,
\end{equation}
where $\bar{\mu}$ is defined as in Lemma \ref{Lions}.
\end{lemma}
\begin{proof}
By assumption we can for every $\varepsilon>0$ find a number $N\in \N$ and sequences of positive functions $\lbrace \rho_n^{(1)}\rbrace_n$ and $\lbrace \rho_n^{(2)}\rbrace_n$ satisfying the properties \eqref{eq: dichotomy}. In addition, we may assume {(see \cite{lions})} that $\lbrace \rho_n^{(1)}\rbrace_n$ and $\lbrace \rho_n^{(2)}\rbrace_n$ satisfy
\begin{align*}
& \mathrm{supp} \, \rho_n^{(1)}\subset (y_n-R_n,y_n+R_n), \\
&\mathrm{supp} \, \rho_n^{(2)}\subset (-\infty, y_n-2R_n)\cup (y_n+2R_n, \infty),
\end{align*}
where $y_n\in \R$ and $R_n\rightarrow \infty$. Then
\begin{equation}
\label{eq: some bound on rho}
\int_{R_n\leq |x-y_n|\leq 2R_n}\rho_n\dx\leq \varepsilon.
\end{equation}

Choose $\varphi,\psi$ as in Theorem \ref{thm: dichotomy convergence}, satisfying $\varphi^2+\psi^2=1$ in addition, and define
\begin{align*}
 \varphi_n(x) & =\varphi((x-y_n)/R_n), \\
 \psi_n(x) & =\psi((x-y_n)/R_n)
\end{align*}
and set $u_n^{(1)}=\varphi_nu_n$, $u_n^{(2)}=\psi_n u_n$. Since $\U(u_n^{(1)})$ is uniformly bounded for all $n\in \N$, there exists a subsequence of $\lbrace u_n^{(1)}\rbrace_n$, still denoted $\lbrace u_n^{(1)}\rbrace_n$, and a real number $\bar{\lambda}=\bar{\lambda}(\varepsilon)$ such that $\U(u_n^{(1)})\rightarrow\bar{\lambda}$. This implies that \eqref{eq: a Z} holds for sufficiently large $n$.

To prove \eqref{eq: b Z}, we write
\begin{align}
\label{eq: proof of dichotomy Z}
\int_{\R}F(u_n)\dx &=\int_{|x-y_n|\leq R_n}F(u_n)\dx+\int_{|x-y_n|\geq 2R_n}F(u_n)\dx \nonumber \\
&\quad +\int_{R_n\leq|x-y_n|\leq 2R_n}F(u_n)\dx \nonumber \\
&=\int_{|x-y_n|\leq R_n}F(u_n^{(1)})\dx+\int_{|x-y_n|\geq 2R_n}F(u_n^{(2)})\dx \nonumber \\
&\quad +\int_{R_n\leq|x-y_n|\leq 2R_n}F(u_n)\dx \\
&=\int_{\R}F(u_n^{(1)})\dx+\int_{\R}F(u_n^{(2)})\dx \nonumber \\
&\quad +\int_{R_n\leq |x-y_n|\leq 2R_n} F(u_n)-F(u_n^{(1)})-F(u_n^{(2)})\dx \nonumber.
\end{align}
By the Sobolev embedding theorem, \eqref{eq: some bound on rho} implies 
\begin{equation*}
\int_{R_n\leq |x-y_n|\leq 2R_n} |u_n|^{p+1}\dx\leq C \varepsilon^{2/(p+1)},
\end{equation*}
where $C>0$ is independent of $n$. This implies that the last line of \eqref{eq: proof of dichotomy Z} can be made less than $\varepsilon$ by taking $n$ large enough. Thus $|\U(u_n)-\U(u_n^{(1)})-\U(u_n^{(2)})| \leq \varepsilon$ and \eqref{eq: b Z} follows. To prove \eqref{eq: c Z}, note that
\begin{align*}
\J_{\kappa}& (u_n^{(1)})+\J_{\kappa}(u_n^{(2)})= \nonumber\\
= &\int_{\R}\varphi_n^2 u_nL u_n\dx+\int_{\R}\varphi_n u_n(L(\varphi_n u_n)-\varphi_nL u_n)\dx \nonumber \\
&+\int_{\R}\psi_n^2 u_nL u_n\dx+\int_{\R}\psi_n u_n(L(\psi_n u_n)-\psi_nL u_n)\dx \\
&+\int_{\R}(\varphi_n^2+\psi_n^2)\kappa u_n^2 \dx. \nonumber
\end{align*}
By Theorem \ref{thm: dichotomy convergence} and Lemma \ref{lem: lower bound on Hs norm}, 
\begin{equation*}
\left|\int_{\R}\varphi_n u_n(L(\varphi_n u_n)-\varphi_nL u_n)\dx\right|\leq \beta(R_n)\|u_n\|_{H^{s/2}}^2\leq \beta(R_n)M,
\end{equation*}
with an equivalent bound for $\left|\int_{\R}\psi_n u_n(L(\psi_n u_n)-\psi_nL u_n)\dx\right|$. As $\beta(R_n)\rightarrow 0$, $N$ can therefore be chosen sufficiently big so that
\begin{equation*}
\J_{\kappa}(u_n)-\varepsilon \leq \J_{\kappa}(u_n^{(1)})+\J_{\kappa}(u_n^{(2)})\leq \J_{\kappa}(u_n)+\varepsilon,
\end{equation*}
for all $n\geq N$.

To prove \eqref{eq: G(u^1)=mu}, we first write
\begin{align}
\label{eq: decompose tilde L}
\left(L^\frac{1}{2}u_n^{(1)}\right)^2= &\left(L^\frac{1}{2}(\varphi_nu_n)-\varphi_nL^\frac{1}{2}u_n\right)^2
 \\
&+ 2\varphi_nL^\frac{1}{2}u_n \left(L^\frac{1}{2}(\varphi_nu_n)-\varphi_nL^\frac{1}{2}u_n\right)+ \varphi_n^2\left(L^\frac{1}{2}u_n\right)^2. \nonumber
\end{align}
Theorem \ref{thm: dichotomy convergence} holds equally well for $L^\frac{1}{2}$, and so $N$ can be taken sufficiently large so that
\begin{equation*}
\int_{\R}\left(L^\frac{1}{2}u_n^{(1)}\right)^2\dx= \int_{\R}\varphi_n^2\left(L^\frac{1}{2}u_n\right)^2\dx+{\mathcal O}(\varepsilon),
\end{equation*}
for all $n\geq N$. Thus we can write
\begin{align*}
\J_{\kappa}(u_n^{(1)})& \geq \int_{\R} \kappa (u_n^{(1)})^2+\left(L^\frac{1}{2}u_n^{(1)}\right)^2\dx \nonumber \\
&=\int_{\R} \varphi_n^2\rho_n\dx+{\mathcal O}(\varepsilon) \nonumber \\
&=\int_{|x-y_n|\leq R_n}\rho_n \dx +\int_{R_n\leq |x-y_n|\leq 2R_n}\varphi_n^2\rho_n\dx +{\mathcal O}(\varepsilon) \\
&=\int_{\R}\rho_n^{(1)}+{\mathcal O}(\varepsilon) \nonumber \\
&\geq \bar{\mu}+{\mathcal O}(\varepsilon), \nonumber
\end{align*}
where the fourth line follows from \eqref{eq: dichotomy} and our assumptions on the support of $\rho_n^{(1)}$. This proves \eqref{eq: G(u^1)=mu}. To prove \eqref{eq: G(u^2)} we proceed similarly, noting that \eqref{eq: decompose tilde L} holds for $u_n^{(2)}$ and $\psi_n$, and that Theorem \ref{thm: dichotomy convergence} still applies in this case. We get
\begin{align*}
\J_{\kappa}(u_n^{(2)})& \geq \int_{\R} \kappa (u_n^{(2)})^2+\left(L^\frac{1}{2}u_n^{(2)}\right)^2\dx \nonumber \\
&=\int_{\R} \psi_n^2\rho_n\dx+{\mathcal O}(\varepsilon) \nonumber \\
&=\int_{|x-y_n|\geq 2R_n}\rho_n \dx +\int_{R_n\leq |x-y_n|\leq 2R_n}\psi_n^2\rho_n\dx +{\mathcal O}(\varepsilon) \\
&=\int_{\R}\rho_n^{(2)}+{\mathcal O}(\varepsilon) \nonumber \\
&\geq \mu- \bar{\mu}+{\mathcal O}(\varepsilon). \nonumber
\end{align*}
\end{proof}

As a consequence of Lemma \ref{lem: dichotomy Z}, we have the following result from \cite{zeng}:
\begin{lemma}
\label{lem: cons dichotomy Z}
Assume that dichotomy holds for $\rho_n$. Then there exists $\lambda_1\in (0,\lambda)$ such that
\begin{equation*}
\Gamma_\lambda\geq \Gamma_{\lambda_1}+\Gamma_{\lambda-\lambda_1}.
\end{equation*}
\end{lemma}
\begin{proof}
This is Lemma 2.16 in \cite{zeng}, and having established Lemma \ref{lem: dichotomy Z} and Lemma \ref{lem: Zeng 3}, the proof presented there holds without modification in the present case.
\end{proof}

As an immediate consequence of \ref{lem: cons dichotomy Z} we can preclude dichotomy.
\begin{corollary}
\label{cor: no dichotomy Z}
Dichotomy does not occur.
\end{corollary}
\begin{proof}
This follows directly from Lemma \ref{lem: Zeng 4} and Lemma \ref{lem: cons dichotomy Z}.
\end{proof}
With vanishing and dichotomy excluded we are able to state the main existence result of this section:
\begin{lemma}
\label{lem: existence of minimizers Z}
The set $G_\lambda$ is non-empty for every $\lambda>0$. Moreover, for every minimizing sequence $\lbrace u_n\rbrace_n$, there exists a sequence of real numbers $\lbrace y_n\rbrace_n$ such that $\lbrace u_n(\cdot+y_n)\rbrace_n$ has a subsequence that converges in $H^{s/2}$ to an element $w\in G_\lambda$.
\end{lemma}
\begin{proof}
Let $\lbrace u_n\rbrace_n$ be a minimizing sequence. From Lemmas \ref{Lions}, \ref{lem: no vanishing Z} and Corollary \ref{cor: no dichotomy Z} we know that compactness occurs. That is, there exists a subsequence of $\lbrace u_n \rbrace_n$, which we denote by $\lbrace u_n \rbrace_n$, and a sequence of real numbers $\lbrace y_n\rbrace_n$ such that for any $\varepsilon>0$, one can find $R>0$ for which
\begin{equation*}
\int_{|x-y_n|\leq R}\rho_n \dx\geq \mu-\varepsilon
\end{equation*}
for all $n$, which implies
\begin{equation*}
\varepsilon\geq \int_{|x-y_n|\geq R}\rho_n\dx\geq \int_{|x-y_n|\geq R} (\kappa+\inf_{\xi\in \R}m(\xi)) u_n^2\dx.
\end{equation*}
We define $\tilde{u}_n(x)=u_n(x+y_n)$. Then
\begin{equation}
\label{eq: norm less than epsilon}
0\leq (\kappa+\inf_{\xi\in \R}m(\xi)) \int_{|x|\geq R}\tilde{u}_n^2\dx\leq \varepsilon.
\end{equation}
Thus, for every $k\in \N$, there exists an $R_k$ such that
\begin{equation*}
\int_{|x|\geq R_k}\tilde{u}_n^2\dx\leq \frac{1}{k}.
\end{equation*}
Furthermore $\lbrace \tilde{u}_n\rbrace_n$ is bounded in $H^{s/2}$, so for every $k\in \N$ there exists a $w_k\in L^2([-R_k, R_k]^c)$ and a subsequence of $\lbrace \tilde{u}_{n}\rbrace_n$, denoted $\lbrace \tilde{u}_{k,n}\rbrace_n$, such that $\tilde{u}_{k,n}\rightarrow w_k$ in $L^2([-R_k, R_k]^c)$ and
\begin{equation}
\int_{|x|\geq R_k}\tilde{u}_{k,n}^2\dx\leq \frac{1}{k}
\end{equation}
for all $n$. A Cantor diagonalization argument on the sequences $\lbrace \tilde{u}_{k,n}\rbrace_n$ yields a subsequence of $\lbrace \tilde{u}_{n}\rbrace_n$, still denoted by $\lbrace \tilde{u}_{n}\rbrace_n$, that converges strongly in $L^2$ to some function $w\in L^2$. Furthermore, $\lbrace \tilde{u}_{n}\rbrace_n$ converges weakly in $H^{s/2}$ by the Banach-Alaoglu theorem, and so $w\in H^{s/2}$ as well. The $L^2$ and weak $H^{s/2}$ convergence implies $L^{p+1}$ convergence:
\begin{align}
\label{eq: convergence in L^(p+1)}
\|\tilde{u}_n-w\|_{L^{p+1}}&\leq C\|\tilde{u}_n-w\|_{H^{(p-1)/2(p+1)}} \nonumber \\
&\leq C\|\tilde{u}_n-w\|_{H^0}^{((p+1)s-p+1)/(s(p+1))}\|\tilde{u}_n-w\|_{H^{s/2}}^{(p-1)/(s(p+1)} \\
&\leq C'\|\tilde{u}_n-w\|_{L^2}^{((p+1)s-p+1)/(s(p+1))}. \nonumber
\end{align}
Recall that we are assuming $1<p<\frac{1+s}{1-s}$, and so $\frac{p-1}{2(p+1)}<\frac{s}{2}$, which makes the above use of the Sobolev interpolation inequality valid.

Thus 
\begin{equation}
\label{eq: F(g)=q}
\U(w)=\lambda.
\end{equation}
By weak lower semi-continuity of the Hilbert norm we have
\begin{equation*}
\J_{\kappa}(w)\leq \liminf \J_{\kappa}(\tilde{u}_{n})=\Gamma_\lambda,
\end{equation*}
hence, by \eqref{eq: F(g)=q} and the definition of $\Gamma_\lambda$,
\begin{equation*}
\J_{\kappa}(w)=\Gamma_\lambda
\end{equation*}
and $w\in G_\lambda$. The above equations and remarks imply $\tilde{u}_{n}\rightarrow w$ in $H^{s/2}$.
\end{proof}
Now observe that if $f$ satisfies (B2), then $\lbrace u_n\rbrace_n$ is minimizing sequence for $\Gamma_\lambda$ if and only if $\lbrace -u_n\rbrace_n$ is a minimizing sequence for $\Gamma_{-\lambda}$. Recalling the calculations regarding the Lagrange multiplier at the beginning of the section and the scalings \eqref{eq: scale 1} and \eqref{eq: scale 2}, we have proven Theorem \ref{thm: existence} (ii).

\section{Concentration-compactness for \eqref{eq: constrained var. prob. 1}}
\label{existence2}

In this section we will prove existence of minimizers of
\begin{equation*}
I_q:=\inf\lbrace \E(w) : w\in H^{s/2} \,\, \text{and}\,\, \mathcal{Q}(w)=q\rbrace,
\end{equation*}
for $q>0$. The basic approach to finding minimizers of $I_q$ is the same as for $\Gamma_\lambda$, and many of the arguments above can be reused and/or modified to the present case. In this section we assume that $p\in (1,2s+1)$.

\begin{lemma}
\label{lem: I_q<0}
For all $q>0$, one has $-\infty<I_q$. Moreover, there exists a number $q_0\geq 0$ such that for all $q>q_0$,
\begin{equation*}
-\infty<I_q<0.
\end{equation*}
If $0\leq m(\xi)\leq A_2|\xi|^s$ for $|\xi|\leq 1$, then the statement holds for $q_0=0$.
\end{lemma}
\begin{proof}
To prove $I_q>-\infty$, we use the Sobolev embedding and interpolation theorems to obtain
\begin{equation*}
\left|\int_{\R}F(\varphi)\dx\right|\leq C\|\varphi\|_{H^{(p-1)/(2(p+1))}}^{p+1}\leq C\|\varphi\|_{H^0}^{((p+1)s-p+1)/s}\|\varphi\|_{H^{s/2}}^{(p-1)/s}.
\end{equation*}
Thus
\begin{align}
\label{eq: lower bound on E}
\E(\varphi)&=\E(\varphi)+\Q(\varphi)-\Q(\varphi) \nonumber \\
&=\frac{1}{2}\int_{\R}\varphi L\varphi+\varphi^2\dx-\int_{\R}F(\varphi)\dx-q \\
&\geq C_1\|\varphi\|_{H^{s/2}}^2-C_2q^{((p+1)s-p+1)/s}\|\varphi\|_{H^{s/2}}^{(p-1)/s}-q \nonumber
\end{align}
where $C_1,C_2>0$ depend only on the symbol $m$ and the Sobolev embedding constant, respectively. By assumption $(p-1)/s<2$, hence the growth of the term with negative sign in the last line of \eqref{eq: lower bound on E} is bounded by the growth of the positive term, and it follows that $I_q>-\infty$.

To prove that $I_q<0$ for all $q$ big enough,  choose $\varphi\in H^{s/2}$ such that $F(\varphi)$ is non-negative. This can be done by taking $\varphi$ to be non-positive if $c_p<0$, and $\varphi$ non-negative if $c_p>0$. For each $q>0$, there exists a number $a=a(q)>0$ such that $\Q(a(q)\varphi)=q$. Then
\begin{equation*}
I_q\leq \E(a(q)\varphi)=\frac{a(q)^2}{2}\int_\R \varphi L\varphi\dx -a(q)^{p+1}\int_\R F(\varphi)\dx.
\end{equation*}
Note that $a(q)\rightarrow \infty$ as $q\rightarrow \infty$ and that $p+1>2$. Hence the right-hand side in the equation above will be negative for all $q$ large enough. This proves the existence of a $q_0\geq 0$ as stated.

Assume now that, in addition to (A), $0\leq m(\xi)\leq A_2|\xi|^s$ for $|\xi|\leq 1$. Let $q>0$ and again choose $\varphi\in H^{s/2}$ such that $F(\varphi)$ is non-negative and $\Q(\varphi)=q$. For $t>0$, set $\varphi_t(x)=\sqrt{t}\varphi(tx)$. Then $\Q(\varphi_t)=q$ for all $t>0$,
\begin{equation*}
\int_\R F(\varphi_t)\dx=t^{(p-1)/2}\int_\R F(\varphi)\dx,
\end{equation*}
and
\begin{equation*}
\int_\R \varphi_t L\varphi_t\dx=\int_\R m(t\xi)|\widehat{\varphi}(\xi)|^2\dxi\leq t^sA_2\int_\R |\xi|^s|\widehat{\varphi}(\xi)|^2\dxi\leq t^s A_2 \|\varphi\|_{H^{s/2}}^2.
\end{equation*}
As $p\in(1,2s+1)$ by assumption, we get $(p-1)/2<s$, so that $t^s\rightarrow 0$ faster than $t^{(p-1)/2}\rightarrow 0$ as $t\rightarrow 0^+$. As $F(\varphi)$ is non-negative, it follows from the calculations above that for $t>0$ sufficiently small, $\int_\R F(\varphi_t)\dx>\int_\R \varphi_tL\varphi_t\dx$, which implies that $I_q\leq \E(\varphi_t)<0$. As $q>0$ was arbitrary, this proves the statement.
\end{proof}

\begin{lemma}
\label{lem: min seq bounded}
If $\lbrace u_n\rbrace_n$ is a minimizing sequence for $I_q$, then
\begin{itemize}
\item[(i)] $\|u_n\|_{H^{s/2}}\leq K$ for some constant $K>0$ and all $n$,
\item[(ii)] if $I_q<0$, then $\|u_n\|_{p+1}\geq \delta$ for some constant $\delta>0$ and all sufficiently large $n$.
\end{itemize}
In particular, (ii) holds if $q>q_0$ for $q_0$ as in Lemma \ref{lem: I_q<0}.
\end{lemma}
\begin{proof}
By Assumption (A), Pareseval's inequality and the Sobolev embedding and interpolation theorems, we have, for some constant $\theta$ depending only on $m$, that
\begin{align*}
\theta\|u_n\|_{H^{s/2}}^2&\leq \E(u_n)+Q(u_n)+\int_{\R}F(u_n)\dx \nonumber \\
&\leq \sup_{n}\E(u_n)+q+C\|u_n\|_{H^{(p-1)/(2(p+1))}}^{p+1} \\
&\leq C'+Cq^{((p+1)s-p+1)/s}\|u_n\|_{H^{s/2}}^{(p-1)/s}. \nonumber
\end{align*}
By assumption, $(p-1)/2<2$, and so we have bounded the square of $\|u_n\|_{H^{s/2}}$ by a smaller power, and the existence of a bound $K$ follows. To prove statement 2 we argue by contradiction. If no such constant $\delta$ exists, then
\begin{equation*}
\liminf_{n\rightarrow\infty}\int_{\R}F(u_n)\dx\leq 0,
\end{equation*}
which implies
\begin{align*}
I_q&=\lim_{n\rightarrow \infty}\left(\frac{1}{2}\int_{\R}u_nL u_n\dx -\int_{\R}F(u_n)\dx\right) \nonumber \\
&\geq \liminf_{n\rightarrow\infty}\left(-\int_{\R}F(u_n)\dx\right)\geq 0,
\end{align*}
contradicting the assumption that $I_q<0$.
\end{proof}

\begin{lemma}
\label{lem: I_(q1+q2)<I_q1+I_q2}
For all $q_1,q_2>0$ such that $I_{q_1+q_2}<0$, one has
\begin{equation*}
I_{(q_1+q_2)}<I_{q_1}+I_{q_2}.
\end{equation*}
\end{lemma}
\begin{proof}
We start by claiming that if $q>0$ and $I_q<0$, then for $t>1$
\begin{equation*}
I_{tq}<tI_q.
\end{equation*}
To see this let $\lbrace u_n\rbrace$ be a minimizing sequence for $I_q$ and define $\tilde{u}_n=\sqrt{t}u_n$ for all $n$, so that $\Q(\tilde{u}_n)=tq$ and hence $\E(\tilde{u}_n)\geq I_{tq}$ for all $n$. Then for all $n$ we have
\begin{equation*}
I_{tq}\leq \frac{1}{2}\int_{\R}\tilde{u}_nL\tilde{u}_n\dx-\int_{\R}F(\tilde{u}_n)\dx=t\E(u_n)+(t-t^{(p+1)/2})\int_{\R}F(u_n)\dx.
\end{equation*}
Now taking $n\rightarrow \infty$ and using Lemma \ref{lem: min seq bounded}, we obtain
\begin{equation*}
I_{tq}\leq tI_q+\frac{1}{2}(t-t^{(p+1)/2})\delta<tI_q
\end{equation*}
since $p+1>2$ and $t>1$. If $I_{q_1},I_{q_2}\geq 0$ the statement is trivial. Assume therefore that, say, $I_{q_1}<0$. Then the claim above holds for $I_{q_1}$ and we can argue as in the proof of Lemma \ref{lem: Zeng 4}.
\end{proof}

We note that for any minimizing sequence $\lbrace u_n\rbrace_n$ of $I_q$, the sequence $\lbrace \frac{1}{2}u_n^2\rbrace_n$ satisfies the conditions for Lemma \ref{Lions} with $\mu=q$. We will now preclude vanishing and dichotomy, and thereby prove that compactness occurs.

\begin{lemma}
\label{lem: vanishing}
If $I_q<0$, vanishing does not occur.
\end{lemma}
\begin{proof}
From Lemmas \ref{lem: min seq bounded} and \ref{lem: lower bound on p+1 norm} we conclude that there exists a number $\eta>0$ and an integer $N\in \N$ such that
\begin{equation*}
\sup_{y\in \R}\int_{y-2}^{y+2}|u_n|^{p+1}\dx\geq \eta
\end{equation*}
for all $n>N$. The result now follows from standard embedding and interpolation arguments.
\end{proof}

\begin{lemma}
\label{lem:dichotomy}
Assume dichotomy occurs. Then for each $\varepsilon>0$ there is a subsequence of $\lbrace u_n\rbrace_n$, still denoted $\lbrace u_n\rbrace_n$, a real number $\bar{q}\in (0, q)$, $N\in \N$ and two sequences $\lbrace u_n^{(1)}\rbrace_n, \lbrace u_n^{(2)}\rbrace_n\subset H^{s/2}$ satisfying for all $n\geq N$:
\begin{subequations}
\begin{align}
& |\Q(u_n^{(1)})-\bar{q}|<\varepsilon, \label{eq: a A} \\
& |\Q(u_n^{(2)})-(q-\bar{q})|<\varepsilon, \label{eq: b A} \\
& \E(u_n)\geq \E(u_n^{(1)})+\E(u_n^{(2)})+\varepsilon. \label{eq: c A}
\end{align}
\end{subequations}
\end{lemma}
\begin{proof}
This proof follows along the lines of the proof of Lemma \ref{lem: dichotomy Z}, but we present it in detail for the sake of clarity as some different argumentation are needed.

As noted in Lemma \ref{lem: dichotomy Z}, we can, by assumption, for $\varepsilon>0$ find a number $N\in \N$ and sequences $\lbrace \rho_n^{(1)}\rbrace_n$ and $\lbrace \rho_n^{(2)}\rbrace_n$ of positive functions satisfying the properties \eqref{eq: dichotomy}, where $\rho_n=\frac{1}{2}u_n^2$, $\mu=q$ and $\bar{q}=\bar{\mu}$. We may assume that $\lbrace \rho_n^{(1)}\rbrace_n$ and $\lbrace \rho_n^{(2)}\rbrace_n$ satisfy
\begin{align*}
& \mathrm{supp} \, \rho_n^{(1)}\subset (y_n-R_n,y_n+R_n), \\
&\mathrm{supp} \, \rho_n^{(2)}\subset (-\infty, y_n-2R_n)\cup (y_n+2R_n, \infty),
\end{align*}
where $y_n\in \R$ and $R_n\rightarrow \infty$. Then
\begin{equation}
\label{eq: eq 1 dichotomy lemma}
\frac{1}{2}\int_{R_n\leq |x-y_n|\leq 2R_n} u_n^2\dx\leq \varepsilon,
\end{equation}
for all $n\geq N$. Now choose $\varphi,\psi$ as in Theorem \ref{thm: dichotomy convergence}, satisfying $\varphi^2+\psi^2=1$ in addition, and define $\varphi_n(x)=\varphi((x-y_n)/R_n)$, $\psi_n(x)=\psi((x-y_n)/R_n)$, $u_n^{(1)}=\varphi_nu_n$ and $u_n^{(2)}=\psi_n u_n$. By the definitions of $u_n^{(1)}$ and $\rho_n^{(1)}$,
\begin{align*}
\left|\Q\left(u_n^{(1)}\right) - \int_{\R}\rho_n^{(1)} \dx \right|  = & \int_{|x-y_n|\leq R_n} |\frac{1}{2}u_n^2-\rho_n^{(1)}|\dx \\
& +\frac{1}{2}\int_{R_n\leq |x-y_n|\leq 2R_n}\varphi_n^2 u_n^2\dx \\
 \leq & \varepsilon + \frac{1}{2}\int_{R_n\leq |x-y_n|\leq 2R_n} u_n^2\dx\leq 2\varepsilon,
\end{align*}
for all $n\geq N$, where the last inequality follows from \eqref{eq: eq 1 dichotomy lemma}. By definition $\left| \int_{\R}\rho_n^{(1)} \dx -\bar{q}\right| \leq \varepsilon$ and we get \eqref{eq: a A} by repeating the procedure for $\varepsilon/3$. Comparing $\Q\left(u_n^{(2)}\right)$ to $\int_{\R}\rho_n^{(2)} \dx$, we obtain \eqref{eq: b A} by exactly the same arguments. To prove \eqref{eq: c A}, we write
\begin{align*}
\E(u_n^{(1)}) &+\E(u_n^{(2)})=\nonumber \\
= &\frac{1}{2}\left[\int_{\R}\varphi_n^2 u_nL u_n\dx+\int_{\R}\varphi_n u_n(L(\varphi_n u_n)-\varphi_nL u_n)\dx\right] \nonumber \\
&+\frac{1}{2}\left[\int_{\R}\psi_n^2 u_nL u_n\dx+\int_{\R}\psi_n u_n(L(\psi_n u_n)-\psi_nL u_n)\dx\right] \\
&-\int_{\R}\left(\varphi_n^2+\psi_n^2\right)F(u_n)\dx \nonumber \\
&+\int_{\R}\left[(\varphi_n^2-\varphi_n^{p+1})+(\psi_n^2-\psi_n^{p+1})\right]F(u_n)\dx. \nonumber
\end{align*}
It follows from Theorem \ref{thm: dichotomy convergence} and Lemma \ref{lem: min seq bounded} that by taking $n$ sufficiently large (so that $R_n$ is large enough), we get
\begin{equation*}
\E(u_n^{(1)})+\E(u_n^{(2)})\leq \E(u_n)+\varepsilon+\int_{\R}\left[(\varphi_n^2-\varphi_n^{p+1})+(\psi_n^2-\psi_n^{p+1})\right]F(u_n)\dx.
\end{equation*}
For $|x-y_n|\not\in (R_n,2R_n)$ we have, by our choice of $\varphi$ and $\psi$, that $\varphi_n^2=\varphi_r^{p+1}$ and $\psi_r^2=\psi_r^{p+1}$. Thus
\begin{align*}
\left|\int_{\R}\left[(\varphi_n^2-\varphi_n^{p+1})  +(\psi_n^2-\psi_n^{p+1})\right]F(u_n)\dx\right|&\leq \int_{R_n\leq |x-y_n|\leq 2R_n}2 \left| F(u_n)\right|\dx \nonumber \\
&\leq CK^{(p-1)/s}\varepsilon^{((p+1)s-p+1)/s},
\end{align*}
where we used the usual Sobolev embedding and interpolation theorems, combined with the boundedness of $u_n$ in $H^{s/2}$ and \eqref{eq: eq 1 dichotomy lemma}. As $((p+1)s-p+1)/s>0$, this proves that there is an $N$ such that \eqref{eq: c A} holds for all $n\geq N$.
\end{proof}

\begin{corollary}
\label{cor: no dichotomy A}
Dichotomy does not occur.
\end{corollary}
\begin{proof}
Assume to the contrary that dichotomy occurs for a minimizing sequence $\lbrace u_n\rbrace_n$. Then by Lemma \ref{lem:dichotomy} there is for every $\varepsilon>0$, a subsequence of $\lbrace u_n\rbrace_n$, still denoted $\lbrace u_n\rbrace_n$, a number $\bar{q}\in (0,q)$, $N\in \N$ and two sequences $\lbrace u_n^{(1)}\rbrace_n, \lbrace u_n^{(2)}\rbrace_n\subset H^{s/2}$ such that \eqref{eq: a A}-\eqref{eq: c A} are satisfied for all $n\geq N$. Then
\begin{align*}
I_q & =\liminf_{n\rightarrow \infty} \E(u_n) \\
& \geq \liminf_{n\rightarrow \infty} \E(u_n^{(1)})+\E(u_n^{(2)})+\varepsilon \\
& \geq I_{\bar{q}\pm \varepsilon}+I_{(q-\bar{q})\pm \varepsilon} +\varepsilon.
\end{align*}
Taking $\varepsilon\rightarrow 0^+$, we get a contradiction with Lemma \ref{lem: I_(q1+q2)<I_q1+I_q2}.
\end{proof}
Now we are able to present the main existence result of this section.
\begin{lemma}
\label{lem: convergence of minimizing sequences}
Let $\lbrace u_n\rbrace_n$ be a minimizing sequence for $I_q$. If $I_q<0$, then there exists a sequence $\lbrace y_n\rbrace_n\subset\R$ such that the sequence $\lbrace \tilde{u}_n\rbrace_n$ defined by $\tilde{u}_n(x)=u_n(x+y_n)$ has a subsequence that converges in $H^{s/2}$ to a minimizer of $I_q$. In particular, there is a $q_0\geq 0$ such that for all $q>q_0$ the set of minimizers is non-empty.
\end{lemma}

\begin{proof}
Let $\lbrace u_n\rbrace_n$ be a minimizing sequence for $I_q$. By Lemma \ref{lem: vanishing} and Corollary \ref{cor: no dichotomy A} we know that compactness occurs. That is, there is a subsequence of $\lbrace u_n\rbrace_n$, denoted $\lbrace u_n\rbrace_n$, and a sequence $\lbrace y_n\rbrace \subset \R$ such that for every $\varepsilon>0$, there exists $0<r<\infty$ satisfying for all $n\in \N$:
\begin{equation*}
\frac{1}{2}\int_{|x-y_n|\leq r} u_n^2\dx \geq q-\varepsilon.
\end{equation*}
This implies that for every $k\in \N$ we can find $r_k\in \R_+$ so that
\begin{equation*}
\frac{1}{2}\int_{|x|\leq r_k} \tilde{u}_n^2\dx \geq q-\frac{1}{k}.
\end{equation*}
Furthermore, by Lemma \ref{lem: min seq bounded} and the Rellich-Kondrachov theorem, for every $k\in \N$, there is a subsequence of $\lbrace \tilde{u}_n\rbrace_n$, denoted $\lbrace \tilde{u}_{k,n} \rbrace_n$, and a function $w_k\in L^2\left( [-r_k,r_k]\right)$ such that $\tilde{u}_{k,n}\rightarrow w_k$ in $ L^2\left( [-r_k,r_k]\right)$. From the inequalities above, we deduce that $\Q(w_k)\geq q-\frac{1}{k}$. Now the arguments in the proof of Lemma \ref{lem: existence of minimizers Z}, involving a Cantor diagonalization argument, can be straightforwardly be applied to the present case. 

Lemma \ref{lem: I_q<0} guarantees the existence of a $q_0\geq 0$ such that the assumption $I_q<0$ is satisfied for all $q>q_0$.
\end{proof}
It only remains to prove that minimizers of $I_q$ solve \eqref{eq: main equation2} with positive wave speed $c$.
\begin{lemma}
If $I_q<0$, any minimizer of $I_q$ is a solution to \eqref{eq: main equation2} with the wave speed $c>0$ being the Lagrange multiplier.
\end{lemma}
\begin{proof}
Let $w\in H^{s/2}$ be a minimizer of $I_q$. Then by the Lagrange multiplier principle there exists $\gamma \in \R$ such that
\begin{equation*}
\E'(w)+\gamma\Q'(w)=0,
\end{equation*}
where $\E'(w)$ and $\Q'(w)$ denote the Fr\'echet derivatives of $\E$ and $\Q$ at $w$. The the Fr\'echet derivatives $\E'(w)$ and $\Q'(w)$ are given by
\begin{align*}
&\E'(w)=L w-f(w) \\
&\Q'(w)=w. \nonumber
\end{align*}
Thus $w$ solves \eqref{eq: main equation2} with $c=\gamma$. Now it remains to prove $c=\gamma>0$. Note first that
\begin{align*}
\frac{\mathrm{d}}{\mathrm{d}\theta}\E(w\theta)\vert_{\theta=1} & =\int_{\R}wL w\dx-(p+1)\int_{\R}F(w)\dx \\
&=2\E(w)-(p-1)\int_{\R}F(w)\dx.
\end{align*}
But $\E(w)=I_q<0$ and $\int_{\R}F(w)\dx>0$, so that
\begin{equation*}
\frac{\mathrm{d}}{\mathrm{d}\theta}\E(w\theta)\vert_{\theta=1}<0.
\end{equation*}
By the definition of Fr\'echet derivative,
\begin{align*}
\frac{\mathrm{d}}{\mathrm{d}\theta}\E(w\theta)\vert_{\theta=1} & =\int_{\R}\E'(w)\cdot \frac{\mathrm{d}}{\mathrm{d}\theta} [w\theta]\vert_{\theta=1}\dx \\
& =-\gamma\int_{\R} \Q'(w)w\dx \\
&=-\gamma\int_{\R}w^2\dx,
\end{align*}
and thus $\gamma>0$.
\end{proof}
 
This concludes the proof of Theorem \ref{thm: existence} (i).

\section{Concentration-compactness for inhomogeneous nonlinearities}
\label{different nonlinearities}

In this section we will prove Theorem \ref{thm: existence} (iii). That is, we will consider nonlinearities of the form $g(u)=u+f(u)$, where $f$ is as in (B). Recall that we have set $\tilde{\U}(u)=\int_{\R} \frac{1}{2}u^2+F(u)\dx$ and 
\begin{equation*}
\tilde{\Gamma}_{\lambda}=\inf \lbrace \J(w) : w\in H^{s/2} \,\, \text{and} \,\, \tilde{\U}(w)=\lambda\rbrace.
\end{equation*}
For a minimizing sequence $\lbrace u_n\rbrace_n$ of $\tilde{\Gamma}_\lambda$ we will define the sequence $\lbrace \rho_n \rbrace_n$ as in Section \ref{existence}. As before, we will show existence of minimizers by precluding vanishing and dichotomy and appeal to Lemma \ref{Lions}. To do so, we need the following Lemma from \cite{zeng}:
\begin{lemma}
\label{lem: lower bound on p+1 norm Z}
There exists a real number $\lambda_0\geq 0$ such that for all $\lambda>\lambda_0$, any minimizing sequence $\lbrace u_n\rbrace_n$ of $\tilde{\Gamma}_\lambda$ satisfies for sufficiently large $n$:
\begin{equation*}
\int_{\R} |u_n|^{p+1}\dx\geq \delta,
\end{equation*}
for some $\delta>0$. If $0\leq m(\xi)\leq A_2|\xi|^s$ for $|\xi|\leq 1$ and $p\in (1,2s+1)$, then the statement holds for $\lambda_0=0$.
\end{lemma}
\begin{proof}
Observe that $\J(u)-\tilde{\U}(u)=\E(u)$ and any minimizing sequence for $\tilde{\Gamma}_\lambda$ is also a minimizing sequence for $\bar{\Gamma}_\lambda=\inf\lbrace \E(u) : u\in H^{s/2} \,\, \text{and} \,\, \tilde{\U}(u)=\lambda\rbrace$. If $\bar{\Gamma}_\lambda<0$ then the result follows from the proof of Lemma \ref{lem: min seq bounded} (ii). That there exists $\lambda_0\geq 0$ such that $\bar{\Gamma}_\lambda<0$ for all $\lambda>\lambda_0$ can be proved exactly in the same way as $I_q<0$ for $q>q_0$ was proved in Lemma \ref{lem: I_q<0}: Choose $\varphi\in H^{s/2}$ such that $F(\varphi)$ is non-negative and let $a(\lambda)>0$ be such that $\tilde{\U}(a(\lambda)\varphi)=\lambda$ for $\lambda>0$. Then $a(\lambda)\rightarrow \infty$ as $\lambda\rightarrow \infty$. This proves the existence of a $\lambda_0\geq 0$ as in the statement.

Assume now that $0\leq m(\xi)\leq A_2|\xi|^s$ and let $1<p<2s+1$. Let $q>0$. From the proof of Lemma \ref{lem: I_q<0}, we know that for any $\varepsilon>0$, we can find $\varphi\in H^{s/2}$ such that $\Q(\varphi)=q$, $F(\varphi)$ is non-negative, $\E(\varphi)<0$ and $\int_\R F(\varphi)\dx<\varepsilon$. The latter inequality implies that $|\tilde{\U}(\varphi)-q|<\varepsilon$. As $\varepsilon>0$ and $q>0$ were arbitrary, this shows that $\bar{\Gamma}_\lambda<0$ for every $\lambda>0$ when $1<p<2s+1$.
\end{proof}
Lemma \ref{lem: lower bound on p+1 norm Z} illustrates the difference between homogeneous and inhomogeneous nonlinearities in \eqref{eq: main equation 2}, and that there is a change in behaviour for the inhomogeneous case at the critical exponent $p=2s+1$.
\begin{lemma}
Vanishing does not occur.
\end{lemma}
\begin{proof}
Having established Lemma \ref{lem: lower bound on p+1 norm Z} the conditions in Lemma \ref{lem: lower bound on p+1 norm} are satisfied and the result follows from the proof of Lemma \ref{lem: no vanishing Z}.
\end{proof}

\begin{lemma}
\label{lem: T_tq<tT_q Z}
If $\lambda>\lambda_0$ and $\theta>1$ then $\tilde{\Gamma}_{\theta \lambda}<\theta \tilde{\Gamma}_\lambda$
\end{lemma}
\begin{proof}
Let $\lbrace u_n\rbrace_n$ be a minimizing sequence for $\tilde{\Gamma}_\lambda$. Choose $\alpha_n>0$ such that $\tilde{\U}(\alpha_n u_n)=\theta \lambda$. Since $\tilde{\U}(u_n)=\lambda$ we find
\begin{equation*}
\alpha_n^2=\theta-\frac{\alpha_n^2(\alpha_n^{p-1}-1)}{\lambda}\int_{\R}F(u_n)\dx.
\end{equation*}
Thus
\begin{align*}
\tilde{\Gamma}_{\theta \lambda}&\leq \J(\alpha_n u_n)=\alpha_n^2\J(u_n) \nonumber \\
&=\left(\theta-\frac{\alpha_n^2(\alpha_n^{p-1}-1)}{\lambda}\int_{\R}F(u_n)\dx\right)\J(u_n).
\end{align*}
Considering the proof of Lemma \ref{lem: lower bound on p+1 norm Z} we see that the statement of that lemma is true with $F(u_n)$ replacing $|u|^{p+1}$. Thus for there is a $\delta>0$ such that $\int_{\R}F(u_n)\dx\geq\delta$ for all sufficiently large $n$. Furthermore, since $\theta>1$ and $\tilde{\U}(\alpha_n u_n)=\theta \lambda$ it is clear that $\alpha_n\geq 1+\varepsilon$ for some $\varepsilon>0$ for all sufficiently large $n$. It follows that
\begin{equation*}
\theta-\frac{\alpha_n^2(\alpha_n^{p-1}-1)}{\lambda}\int_{\R}F(u_n)\dx<\theta,
\end{equation*}
and so
\begin{equation*}
\tilde{\Gamma}_{\theta \lambda}\leq \alpha_n^2\J(u_n)<\theta \tilde{\Gamma}_\lambda.
\end{equation*}
\end{proof}

\begin{lemma}
\label{lem: T_q<T_q1+T_q2}
If $\lambda>\lambda_0$, $\lambda_1>0$, $\lambda_2>0$ and $\lambda_1+\lambda_2=\lambda$, then $\tilde{\Gamma}_\lambda<\tilde{\Gamma}_{\lambda_1}+\tilde{\Gamma}_{\lambda_2}$.
\end{lemma}
\begin{proof}
Observe that $\bar{\Gamma}_\lambda=\tilde{\Gamma}_\lambda-\lambda$ and it is sufficient to prove that $\bar{\Gamma}_\lambda<\bar{\Gamma}_{\lambda_1}+\bar{\Gamma}_{\lambda_2}$. 

From the proof of Lemma \ref{lem: lower bound on p+1 norm Z} we conclude that $\bar{\Gamma}_\lambda<0$. Hence if $\bar{\Gamma}_{\lambda_1}, \bar{\Gamma}_{\lambda_2}\geq 0$ the claim is trivial. Assume therefore that, say, $\bar{\Gamma}_{\lambda_1}<0$. Then Lemma \ref{lem: T_tq<tT_q Z} holds for $\bar{\Gamma}_{\lambda_1}$. From here the result can be proven in the same fashion as in Lemma \ref{lem: I_(q1+q2)<I_q1+I_q2} or equation \eqref{eq: Gamma_a<Gamma_b+Gamma_(a-b)}.
\end{proof}

Lemma \ref{lem: dichotomy Z} and Lemma \ref{lem: cons dichotomy Z} carry over to the present case straightforwardly, and from Lemma \ref{lem: T_q<T_q1+T_q2} we then conclude that dichotomy does not occur. The arguments in Lemma \ref{lem: existence of minimizers Z} then gives the existence of minimizers. This concludes the proof of Theorem \ref{thm: existence} (iii).

\section{Stability and Regularity}
\label{stable regular}
In this section we will prove Theorem \ref{thm: stability} and discuss the regularity of solutions to equations \eqref{eq: main equation2}, \eqref{eq: main equationZ} and \eqref{eq: secondary eq 2}.
\begin{proof}[Proof of Theorem \ref{thm: stability}]
We prove it only for $D_q$; the proofs for $G_\lambda$ and $\tilde{G}_\lambda$ are equivalent. Assume the statement is false. That is, there exist a number $\varepsilon>0$, a sequence $\lbrace \varphi_n\rbrace_n\subset H^{s/2}$, and a sequence of times $\lbrace t_n\rbrace_n\subset \R$ such that
\begin{equation*}
\inf_{w\in D_q}\|\varphi_n - w\|_{H^{s/2}}<\frac{1}{n}
\end{equation*}
and
\begin{equation*}
\inf_{w\in D_q} \|u_n(\cdot, t_n)-w\|_{H^{s/2}}>\varepsilon
\end{equation*}
for all $n$, where $u_n(x,t)$ solves \eqref{eq: main equation} with $u_n(x,0)=\varphi_n(x)$. Since $\varphi_n\rightarrow D_q$ in $H^{s/2}$, $\E(w)=I_q$ and $\Q(w)=q$, we have $\E(\varphi_n)\rightarrow I_q$ and $\Q(\varphi_n)\rightarrow q$.

Choose a sequence $\lbrace \alpha_n\rbrace_n \subset \R$ such that $\Q(\alpha_n \varphi_n)=q$ for all $n\in \N$. Then $\alpha_n\rightarrow 1$. As $\E(u)$ and $\Q(u)$ are independent of $t$ if $u$ solves \eqref{eq: main equation}, the sequence $f_n:=\alpha_n u_n(\cdot, t_n)$ satisfies $\Q(f_n)=q$ for all $n$ and
\begin{equation*}
\lim_{n\rightarrow \infty} \E(f_n)=\lim_{n\rightarrow \infty} \E(u_n(\cdot, t_n))= \lim_{n\rightarrow \infty} \E(\varphi_n)=I_q.
\end{equation*}
The sequence $\lbrace f_n\rbrace_n$ is therefore a minimizing sequence for $\E$ and from Theorem \ref{thm: existence} and the translation invariance of the functionals $\E$ and $\Q$ we deduce that there is an $N\in \N$ such that for $n\geq N$, there exists $w_n\in D_q$ satisfying
\begin{equation*}
\|f_n-w_n\|_{H^{s/2}}\leq \frac{\varepsilon}{2}.
\end{equation*}
So for $n\geq N$ we have
\begin{align*}
\varepsilon &\leq \|u_n(\cdot,t_n)-w_n\|_{H^{s/2}} \\
&\leq \|u_n(\cdot,t_n)-f_n\|_{H^{s/2}}+\|f_n-w_n\|_{H^{s/2}} \\
&\leq |1-\alpha_n|\|u_n(\cdot,t_n)\|_{H^{s/2}}+\frac{\varepsilon}{2}.
\end{align*}
Taking $n\rightarrow \infty$ we get $\varepsilon\leq \varepsilon/2$, a contradiction.
\end{proof}
Next we consider regularity of the solutions. Since $L : H^{r}\rightarrow H^{r-s}$ for all $r\in \R$, $L u\in H^{-s/2}$ for $u\in H^{s/2}$ and the solutions we have found may be only distributional solutions of \eqref{eq: main equation2} and \eqref{eq: main equationZ}. However, the solutions inherit regularity from the equations themselves.
\begin{lemma}
\label{lem: regularity}
If $f$ satisfies (B) with $p\in (1, \frac{1+s}{1-s})$, then any solution $u\in H^{s/2}$ of \eqref{eq: main equation2}, \eqref{eq: main equationZ} or \eqref{eq: secondary eq 2} is in $L^\infty$.
\end{lemma}
\begin{proof}
Let $p\in (1, \frac{1+s}{1-s})$ be given. Applying the Fourier transform on both sides of \eqref{eq: main equation2}, \eqref{eq: main equationZ} or \eqref{eq: secondary eq 2} yields
\begin{equation*}
(k+m(\xi))\widehat{u}=\widehat{f(u)},
\end{equation*}
for some constant $k>0$ depending on the wave speed $c$. Without loss of generality, we may assume $k=1$, and so
\begin{equation*}
\widehat{u}=\frac{\widehat{f(u)}}{1+m(\xi)}.
\end{equation*}
Since $u\in H^{s/2}$ we have $u\in L^q$ and $f(u)\in L^{\frac{q}{p}}$ for all $q\in [2,\frac{2}{1-s}]$. It follows that $\widehat{f(u)}\in L^q$ for all $\frac{2}{2-p(1-s)}\leq q\leq \infty$. By assumption (A), $(1+m(\cdot))^{-1}\in L^q$ for all $q>1/s$. Let $\varepsilon>0$ be a number such that $\frac{2}{1+s+\varepsilon}\geq 1$. We have
\begin{equation}
\label{eq: Lp iterative inequalities}
\|\widehat{u}\|_{L^{2/(1+s+\varepsilon)}}^{2/(1+s+\varepsilon)}=\|\widehat{u}^{\frac{2}{1+s+\varepsilon}}\|_{L^1}\leq \|\widehat{f(u)}^{\frac{2}{1+s+\varepsilon}}\|_{L^r}\|(1+m(\cdot))^{-\frac{2}{1+s+\varepsilon}}\|_{L^{r'}},
\end{equation}
where $r,r'>0$ are such that $1/r+1/{r'}=1$. We choose the smallest $r$ for which we can guarantee the first term on the right-hand side is finite. That is $\frac{2r}{1+s+\varepsilon}=\frac{2}{2-p(1-s)}$, which gives $r=\frac{1+s+\varepsilon}{2-p(1-s)}$. Then $r'=\frac{1+s+\varepsilon}{(1-s)(p-1)+\varepsilon}$. In order for $\|(1+m(\cdot))^{-\frac{2}{1+s+\varepsilon}}\|_{L^{r'}}$ to be finite we need
\begin{equation*}
\frac{2}{(1-s)(p-1)+\varepsilon}>\frac{1}{s},
\end{equation*}
which gives the inequality $(1-s)(p-1)+\varepsilon<2s$ to be satisfied. We claim that $\varepsilon>0$ can always be chosen such that this holds. For $p$ minimal, i.e. close to $1$, $\varepsilon\leq s$ is sufficient. If $p$ is large, let $\delta$ be the distance from $p$ to $\frac{1+s}{1-s}$. Then we get $2s+\varepsilon-(1-s)\delta<2s$, which is satisfied when $\varepsilon<(1-s)\delta$. Assuming $\varepsilon>0$ is appropriately chosen we get $\widehat{u}\in L^{2/(1+s+\varepsilon)}$ which implies $u\in L^{2/(1-s-\varepsilon)}$, and so $f(u)\in L^q$ for $1\leq q\leq \frac{2}{p(1-s-\varepsilon)}$ and $\widehat{f(u)}\in L^q$ for $\frac{2}{2-p(1-s-\varepsilon)}\leq q\leq \infty$. Inserting $2\varepsilon$ instead of $\varepsilon$ in \eqref{eq: Lp iterative inequalities} and repeating the procedure by choosing $r$ minimal with respect to the new lower bound on $q$ for the $\widehat{f(u)}\in L^q$ we get the inequality $(1-s)(p-1)+\varepsilon(2-p)<2s$ to be satisfied, which is already guaranteed by the first step since $2-p<1$. In general we get the inequality $(1-s)(p-1)+\varepsilon(n+1-np)<2s$ to be satisfied after $n$ iterations. By iterating the procedure enough times, we get $\widehat{u}\in L^1$ and thus $u\in L^\infty$.
\end{proof}

\begin{theorem}
\label{thm: regularity}
If $f$ satisfies (B) with $p\in (1, \frac{1+s}{1-s})$, then any solution $u\in H^{s/2}$ of \eqref{eq: main equation2}, \eqref{eq: main equationZ} or \eqref{eq: secondary eq 2} is in $H^{s}$.
\end{theorem}
\begin{proof}
We apply the Fourier transform on both sides of \eqref{eq: main equation2}, \eqref{eq: main equationZ} or \eqref{eq: secondary eq 2} and obtain
\begin{equation}
\label{eq: regularity}
(k+m(\xi))\widehat{u}=\widehat{f(u)},
\end{equation}
for some $k>0$. As in Lemma \ref{lem: regularity} we assume, without loss of generality, that $k=1$ and furthermore that $c_p=1$. By Plancherel's Theorem and Lemma \ref{lem: regularity} we have
\begin{equation*}
\|\widehat{f(u)}\|_{L^2}=\|f(u)\|_{L^2}\leq \|u\|_{L^\infty}^{p-1}\|u\|_{L^2}<\infty,
\end{equation*}
and so, by \eqref{eq: regularity}, $(1+m(\xi))\widehat{u}\in L^2$ which implies $u\in H^{s}$.
\end{proof}

\appendix
\section*{Appendix: The necessity of a continuity assumption on the symbol $m$}
\label{appendix}

In this appendix we prove that assuming only that $m$ satisfies \eqref{assumption on m}, the existence of solitary wave solutions of either \eqref{eq: main equation} or \eqref{eq: main equation 2} cannot in general be proved using the method of concentration compactness as in the previous sections. This will be done by providing a counter example, where Dichotomy cannot be precluded for minimizing sequences.

Let $E\subset \R_+$ be a closed nowhere dense set of non-zero measure such that every point of $E$ is a limit point of the set. Such a set can be made by constructing a fat Cantor set $C\subset [0,1]$ with measure $1-\alpha\in (0,1)$ and let $E$ be a union of disjoint translates of $C$. To be precise, we first remove from $[0,1]$ the open interval with centre $1/2$ and length $\alpha/2$. In the next step, we remove from each of the two remaining closed intervals the middle open interval of length $\alpha/8$. At the n-th step one removes from each remaining closed interval the middle open intervals of length $\alpha/(2^{2n-1})$. What remains in the limit is called the (fat) Cantor set of measure $1-\alpha$. If we define $m$ by
\begin{equation*}
m(\xi):=\left\{
	\begin{array}{l l}
		A_2|\xi|^s, & \quad \text{if}\,\, \xi\in \R\setminus E,\\
		A_1|\xi|^s, & \quad \text{if}\,\, \xi\in E
	\end{array} \right.
\end{equation*}
for any $s>0$, then $m$ satisfies \eqref{assumption on m}. Assume that $u\in \mathcal{S}(\R)$, the space of smooth, rapidly decreasing functions. Let $\varphi_r$ be as in Theorem \ref{thm: dichotomy convergence}. Then $L(\varphi_ru)-\varphi_rL u\in L^2$ and thus
\begin{align*}
\int_{\R}\varphi_r u(L(\varphi_ru) & -\varphi_rL u)\dx\\
& =\int_{\R}\left(m(\xi)(\widehat{\varphi_r}\ast \widehat{u})(\xi)-(\widehat{\varphi_r}\ast m\widehat{u})(\xi)\right)\overline{\widehat{\varphi_r u}}(\xi)\dxi.
\end{align*}
Using standard properties of the Fourier transform and convolutions, the right hand side can be written as
\begin{equation*}
\int_{\R}\int_{\R}\widehat{\varphi_r}(t)\widehat{u}(\xi-t)(m(\xi)-m(\xi-t))\, \mathrm{d}t\overline{\widehat{\varphi_r u}}(\xi)\dxi.
\end{equation*}
Consider now $\xi\in E$. For every $r>0$, the set $E^c\cap [\xi-1/r,\xi+1/r]$ has non-zero measure, where $E^c=\R\setminus E$. This can be seen by taking any point $x\in E^c\cap[\xi-1/r,\xi+1/r]$ (such a point must exist since the closure of $E$ has empty interior); if there exists no interval 
\begin{equation*}
(x-\varepsilon,x+\varepsilon)\subset E^c\cap[\xi-1/r,\xi+1/r]
\end{equation*}
then $x$ is in the closure of $E$ and, moreover, $(\xi-1/r,\xi+1/r)\subset \overline{E}$, which is a contradiction since $E$ is nowhere dense. Let $\mathcal{L}$ be the Lebesgue measure on the real line. We claim that there exists a constant $k=k(\alpha)\in (0,2)$ independent of $\xi$ and $r$ such that
\begin{equation*}
\mathcal{L}\left( E^c\cap [\xi-1/r,\xi+1/r]\right) \geq k(\alpha)/r \,\, \text{for all} \,\, \xi\in E \,\, \text{and all} \,\, r>1.
\end{equation*}
This means the ratio
\begin{equation*}
\frac{\mathcal{L}\left( E^c\cap [\xi-1/r,\xi+1/r]\right)}{\mathcal{L}\left( [\xi-1/r, \xi+1/r\right)}
\end{equation*}
is bounded from below; i.e. a minimum portion of the interval in the sense of measure always belongs to $E^c$. To see that this claim is true, let $\mathcal{P}$ be the collection of all open intervals deleted from $[0,1]$ in the construction of the Cantor set and their translates according the definition of $E$. Let $\mathcal{P'}$ denote the corresponding collection of closed sets remaining at each step of the construction. Let $\xi$ be given. If there is a $P\in \mathcal{P}$ such that, say, $\mathcal{L}\left( P\cap [\xi-1/r,\xi+1/r]\right)>1/(2r)$ the result is satisfied by $k(\alpha)=1/2$. Assume therefore that this is not the case. By assumption there exists a point $\tau_1\in [\xi+1/(2r),\xi+1/r]$ that is a left endpoint of some $P'\in \mathcal{P'}$. Similarly, there is a point $\tau_2\in [\xi-1/r,\xi-1/(2r)]$ that is the right endpoint of some $P'\in \mathcal{P'}$. Now $[\tau_1,\tau_2]\subseteq [\xi-1/r,\xi+1/r]$ and $\mathcal{L}\left( [\tau_1,\tau_2]\right)>1/r$. Furthermore, $[\tau_1,\tau_2]$ can by construction be written as a union of disjoint elements of $\mathcal{P'}$ and $\mathcal{P}$. Clearly, $\mathcal{L}\left( E^c\cap P\right)=\mathcal{L}(P)$ and by the construction of the Cantor set, $\mathcal{L}\left(E^c\cap P'\right)=\alpha \mathcal{L}(P')$. Thus $\mathcal{L}\left( E^c\cap [\tau_1,\tau_2]\right)>\alpha/r$ and therefore our claim holds with $k(\alpha)=\alpha$.

Observe next that $m(\xi)-m(\xi-t)\leq -(A_2-A_1)|\xi|^s+A_2|1/r|^s$ when $t\in (E^c-\xi)\cap[-1/r,1/r]$. Furthermore, we may assume that $\widehat{u}\geq M$ for some $M>0$ in all intervals $[\xi-1/r,\xi+1/r]$ where $\xi\in E$ for $r$ sufficiently large. Thus
\begin{align*}
\int_{(E^c-\xi)\cap[-1/r,1/r]}& \widehat{\varphi_r}(t)\widehat{u}(\xi-t)(m(\xi-t)-m(\xi))\, \mathrm{d}t \nonumber \\
&\geq M(A_2-A_1)|\xi|^s\int_{(E^c-\xi)\cap[-1/r,1/r]}\widehat{\varphi_r}(t)\, \mathrm{d}t \nonumber \\
&=M(A_2-A_1)|\xi|^s\int_{r\left((E^c-\xi)\cap[-1/r,1/r]\right)}\widehat{\varphi}(x)\dx \nonumber \\
&\geq N(A_2-A_1)|\xi|^s,
\end{align*}
for some constant $N>0$ that does not depend on $r$. This follows since the measure of the area of integration in line three is greater than $k(\alpha)$ and so the integral of $\widehat{\varphi}$ over the set is non-zero. Hence
\begin{align*}
\liminf_{r\rightarrow\infty} &\left|\int_{E}\int_{(E^c-\xi)\cap[-1/r,1/r]} \widehat{\varphi_r}(t)\widehat{u}(\xi-t)(m(\xi)-m(\xi-t))\, \mathrm{d}t\overline{\widehat{\varphi_r u}}(\xi)\dxi\right| \nonumber \\
& \geq N(A_2-A_1)\left|\int_{E}|\xi|^s\overline{\widehat{\varphi_r u}}(\xi)\dxi\right|>0.
\end{align*}
Assume now that dichotomy occurs for a minimizing sequence $\lbrace u_n\rbrace_n$. In Lions' \cite{lions} construction of the sequences in the case of dichotomy $u_n^{(1)}=\varphi_{R_1}u_n$ and $u_n^{(2)}=\psi_{r_n}u_n$ for suitably large $R_1$ and $r_n\rightarrow \infty$, where $\varphi_r$, $\psi_r$ are defined as in Theorem \ref{thm: dichotomy convergence}. We have
\begin{align*}
\J(u_n)= & \J(u_n^{(1)})+\J(u_n^{(2)})+\frac{1}{2}\int_{\R}\eta_{r_n}\left(uL u+\kappa u^2\right)\dx \nonumber \\
&-\int_{\R}\varphi_{R_1} u_n(L(\varphi_{R_1}u_n)-\varphi_{R_1}L u_n)\dx \nonumber \\
&-\int_{\R}\psi_{r_n} u_n(L(\psi_{r_n}u_n)-\psi_{r_n}L u_n)\dx,
\end{align*}
where $\eta_{r_n}=\chi_{\R}-\psi_{r_n}^2-\varphi_{R_1}^2$. For any $\varepsilon>0$ the last term can be taken to be smaller than $\varepsilon$ in absolute value by choosing $r_n$ big enough according to Theorem \ref{thm: dichotomy convergence}, as continuity of $m$ was not used in that part of the proof. By assumption, $\int_{\R} \eta_{r_n}|u|^{p+1}\dx\leq \varepsilon$ for $n$ sufficiently large. If $u_n$ is, say, a rapidly decreasing function so that $L u\in L^{(p+1)/p}$ then H\"older's inequality implies that $\frac{1}{2}\int_{\R}\eta_{r_n}\left( c_1 uL u+c_2 u^2\right)\dx$ goes to zero as $\varepsilon$ goes to zero. However, as our calculations above show, the term  $\int_{\R}\varphi_{R_1} u_n(L(\varphi_{R_1}u_n)-\varphi_{R_1}L u_n)\dx$ can in general be bounded away from zero for all choices of $R_1$ for elements $u_n\in H^{s/2}$.


\bibliographystyle{plain}
\bibliography{references}

\end{document}